\documentclass[final,1p,times,twocolumn]{elsarticle}

\usepackage{color}
\usepackage{ifthen}

  \newcommand{\R}{\ensuremath{\mathbb{R}}}
  \newcommand{\N}{\ensuremath{\mathbb{N}}}

  \newcommand{\M}{\ensuremath{{M}}}
  
  \newcommand{\Nc}{\mathcal{N}}

  \newcommand{\T}{\ensuremath{\top}}

\newcommand{\vast}{\bBigg@{4}}
\newcommand{\Vast}{\bBigg@{5}}

\newcommand{\V}[1]{\ensuremath{\mathbf{#1}}}

\newcommand{\norm}[1]{\left|\left| #1 \right|\right|}

\newcommand{\TODO}[1]{ 
\ifx\NOTES\undefined\else
{\color{red} [!]}\footnote{ {\color{red} TODO: #1}}
\fi
}

\newcommand{\NOTE}[1]{ 
\ifx\NOTES\undefined\else
  \footnote{ {\color{blue} NOTE: #1}}  
\fi
}

\newcommand{\ecomment}[1]{ 
\ifx\NOTES\undefined\else 
{\color{blue}[E]}\footnote{ {\color{blue} Erez: #1}}
\fi
}

\newcommand{\mcomment}[1]{ 
\ifx\NOTES\undefined\else
  {\color{green} [M]}\footnote{ {\color{green} Matan: #1}}  
\fi
}

\newcommand{\Pd}[3]{\ifthenelse{\equal{#3}{1}}{\frac{\partial #1}{\partial #2}}{\frac{\partial^{#3} #1}{\partial #2^{#3}}}}

\newcommand{\Vx}{a}
\newcommand{\mysum}[3]{\underset{#1=#2}{\overset{#3}{\sum}}}

\def \NOTES

\usepackage{graphicx}
\usepackage{amssymb}
\usepackage{amsthm}

\usepackage{amsmath,amsfonts,amssymb}
\usepackage{float}
\usepackage{subcaption}
\usepackage{graphicx}
\usepackage[colorlinks]{hyperref}
\usepackage{amsmath}
\usepackage[justification=centering]{caption}
\usepackage{wrapfig}
\usepackage{verbatim}
\usepackage{hhline}

\graphicspath{{../pdf/}{../jpeg/}{/graphs/}}    
\DeclareGraphicsExtensions{.pdf,.jpeg,.png,.jpg}

\newtheorem{thm}{Theorem}
\newtheorem{Lemma}{Theorem}
\newtheorem{Corollary}{Theorem}

\newtheorem{Definition}{Theorem}

\newtheorem{lem}[Lemma]{Lemma}
\newtheorem{crl}[Corollary]{Corollary}

\newtheorem{defin}[Definition]{Definition}

\newproof{pf}{Proof}
\usepackage{lineno}

\newcommand{\Tr}{\ensuremath{{\scriptscriptstyle \top}}}
\newcommand{\Vxh}{\hat{\Vx}}
\newcommand{\Vy}{\ensuremath{{b}}}
\newcommand{\Vxt}{\ensuremath{{\tilde{\Vx}}}}
\newcommand{\Mx}{\ensuremath{{X}}}
\newcommand{\My}{\ensuremath{{Y}}}

\newcommand{\Mxh}{\ensuremath{\hat{X}}}

\newcommand{\Sx}{\ensuremath{x}}
\newcommand{\Shy}{\ensuremath{y}}
\newcommand{\Sxh}{\ensuremath{\hat{x}}}

\newcommand{\MDSp}{MDS+}

\makeatletter
\def\ps@pprintTitle{%
	\let\@oddhead\@empty
	\let\@evenhead\@empty
	\def\@oddfoot{}%
	\let\@evenfoot\@oddfoot}
\makeatother
\begin{document}

	\begin{frontmatter}
		
		\title{
			Multidimensional Scaling of \\ Noisy High Dimensional Data }

		\author{Erez Peterfreund\corref{corErez}\fnref{label2}}
		\author{Matan Gavish\corref{corMatan}\fnref{label2}}
		\fntext[label2]{{School of Computer Science and Engineering, The Hebrew University, 
		Jerusalem, Israel}}
		\address{The Hebrew University of Jerusalem\fnref{label3}}

		\cortext[corErez]{corresponding author: erezpeter@cs.huji.ac.il}
		\cortext[corMatan]{corresponding author: gavish@cs.huji.ac.il}
		
		\begin{abstract}

		Multidimensional Scaling (MDS) is a classical technique for
		embedding data in low dimensions, still in widespread use today.
		Originally introduced in the 1950's, MDS was not designed with
		high-dimensional data in mind; while it remains popular with
		data analysis practitioners, no doubt it should be adapted to
		the high-dimensional data regime.  In this paper we study MDS
		under modern setting, and specifically, high dimensions and
		ambient measurement noise. We show that, as the ambient noise
		level increase, MDS suffers a sharp breakdown that depends on
		the data dimension and noise level, and derive an explicit
		formula for this breakdown point in the case of white noise. We
		then introduce \MDSp, an extremely simple variant of MDS, which
		applies a carefully derived shrinkage nonlinearity to the
		eigenvalues of the MDS similarity matrix.  Under a loss function
		measuring the embedding quality, \MDSp~ is the unique
		asymptotically optimal shrinkage function.  We prove that \MDSp~
		offers improved embedding, sometimes significantly so, compared
		with classical MDS.  Furthermore, \MDSp~ does not require
		external estimates of the embedding dimension (a famous
		    difficulty in classical MDS), as it calculates the optimal
		    dimension into which the data should be embedded.
			
		\end{abstract}
		
		
	\end{frontmatter}

	\section{Introduction}
	\label{S:1}
	
	Manifold learning and metric learning methods have become central items in
	the toolbox of any modern data scientist. These techniques seek to
	reconstruct a global, low-dimensional geometric structure of a dataset from
	pairwise similarity measurements \cite{tenenbaum2000Isomap,	roweis2000LLE, coifman2006diffusion, vincent2008extracting, bishop1998developments,
	bellet2013survey,bengio2013representation}.

	Multidimensional Scaling \cite{Torgerson1952a} (MDS) was the first
	metric learning algorithm proposed, and arguably the 
	one most widely used today.
 It is used extensively for exploratory data analysis, inference and 
 visualization in many science and engineering disciplines, as well
	as in psychology, medicine and the social sciences
	\cite{torgerson1958theory,kruskal1978multidimensional,
		cox2000multidimensional, borg2005modern, young2013multidimensional}.
	
		In the MDS algorithm, one considers an unknown point cloud
		$\V{\Vy_1},\ldots, \V{\Vy_n}\in\R^p$ and assumes that only the distances
		$\Delta_{i,j}=\norm{\V{\Vy_i}-\V{\Vy_j}}^2$ are observable.
		MDS, which aims to reconstruct the global spatial configuration of the
		point cloud, proceeds as follows. 
	
		\begin{enumerate} 
			\item First, form the similarity matrix 
				\begin{eqnarray}
				\label{eq:S} S=-\frac{1}{2}H\cdot \Delta \cdot H\,, 
			\end{eqnarray}
				where $H=I-\frac{1}{n}\V{1}\cdot \V{1}^\T$ is a data-centering
				matrix.  
			\item Next, diagonalize $S$ to form \begin{eqnarray}
				\label{eq:Sdiag} S=U\cdot D\cdot U' \end{eqnarray}	where
				$D=diag(d_1,\ldots,d_n)$ and $U$ is orthogonal with orthonormal
				columns $\V{u}_1,\ldots \V{u}_n$.  
			\item Then, estimate (or guess) the original
				dimension of the point cloud,
				$r=dim\,Span \{\V{\Vy_1},\ldots \V{\Vy_n}\}$.
			
			\item Finally, 
				return the $n$-by-$r$ matrix with
				columns $\sqrt{d_i} \cdot \V{u}_i$ ($i=1,\ldots,r$). 
				Embed the points into $\R^r$
				using the rows of this matrix.  
		\end{enumerate}

		This paper addresses two crucial
		issues that remain open in the practice of MDS on high-dimensional data: the
	effect of ambient noise, and the choice of embedding dimension. As we will
	see, while these issues are seemingly different, they are   in fact very closely
	related. 
	 Let us first elaborate on each issue in turn.

	 \subsection{Choice of embedding dimension} 
	
	While MDS is extremely popular among practitioners in all walks of science,
	there is a decades-old inherent conundrum involved in its use in practice.
	Strangely, the literature offers no systematic method for
	choosing the embedding dimension $r$.  The original paper
	\cite{Torgerson1952a}, as well as numerous authors since, have proposed
	various heuristics for choosing the ``correct'' embedding dimension. In
	fact, recent tutorials such as \cite{hout2013multidimensional}, and even
	the SPSS user's manual\footnote{\url{https://www.ibm.com/support/knowledgecenter/SSLVMB_23.0.0/spss/tutorials/proxscal_data_howto.html}
	. (Accessed 1/1/2018)},
	still offer no systematic method and recommend Cattel's {\em Scree Plot}
	heuristic \cite{Cattell1966}, a 50-year-old method based on subjective
	visual inspection of the data.

	Our first main contribution in this paper is 
	a systematic
	method for choosing $r$, the embedding dimension, from the data.
	The estimator $\hat{r}$ we propose 
	is provably optimal in the asymptotic regime $n,p\to\infty$,
	under a suitable loss function quantifying the embedding quality, and under the assumption of white ambient noise.

	Concretely, Table \ref{concrete:tab} shows the value of the optimal hard
	threshold $\lambda^*$ for MDS, a concept we develop below. To find the
	asymptotically optimal embedding dimension $\hat{r}$ in an MDS problem
	with $n$ vectors in ambient dimension $p$ and white ambient noise 
	with standard deviation $\sigma$, simply proceed as follows. First,
	let $\beta=(n-1)/p$ and find the value $\lambda^*$ from Table
	\ref{concrete:tab} ({\tt Python} and {\tt Matlab} code to evaluate
	$\lambda^*$ exactly is provided in the code supplement \cite{CODE},
    based on formula \eqref{eq:SVHTOptimalShrinker} below). 
    Then, let $\hat{r}$ be the number of eigenvalues the matrix $S$ from \eqref{eq:S}  that fall above the threshold 
    $(\sigma\cdot \lambda^*)^2$. If $\sigma$ is unknown, as is often the case, use
    the consistent and robust estimator $\hat{\sigma}$ from \eqref{eq:SigmaEstimation} 
    below instead; an implementation of this estimator is included in the code
    supplement \cite{CODE}.

\begin{table}[h]
	\centering
	\begin{tabular}{| c | c | c | c | c | c | c | c | c | c | c |}
		\hline
		$\beta$ & 0.05 & 0.1 & 0.15 & 0.2 & 0.25 & 0.3 & 0.35 & 0.4 & 0.45 & 5 \\[0.2cm]  \hline
		$\lambda^*$ & 1.301 & 1.393 & 1.467 & 1.531 & 1.588 & 1.639 & 1.688 & 1.733 & 1.775 & 1.816 \\ \hline
		\multicolumn{1}{r}{ } & \multicolumn{1}{r}{ } & \multicolumn{1}{r}{ } & \multicolumn{1}{r}{ } &
		\multicolumn{1}{r}{ } & \multicolumn{1}{r}{ } & \multicolumn{1}{r}{ } & \multicolumn{1}{r}{ } &
				\multicolumn{1}{r}{ } & \multicolumn{1}{r}{ } & \multicolumn{1}{r}{ } 		
		\\ \hline
		$\beta$ & 0.55 & 0.6 & 0.65 & 0.7 & 0.75 & 0.8 & 0.85 & 0.9 & 0.95 & 1 \\[0.2cm]  \hline 
		$\lambda^*$ & 1.854 & 1.891 & 1.927 & 1.962 & 1.995 & 2.028 & 2.059 & 2.09 & 2.12 & 2.149 \\[0.2cm] \hline
	\end{tabular}
	\caption{ 
	    The optimal threshold for MDS - see Section \ref{S:3} below.  
	    The asymptotically optimal embedding dimension $\hat{r}$ is obtained by counting the eigenvalues of the matrix S from \eqref{eq:S} that fall above the threshold $(\sigma\lambda^*)^2$. 
	    If $\sigma$ is unknown, as is often the case, use the consistent and robust estimator $\hat{\sigma}$ 
	    from \eqref{eq:SigmaEstimation}.  
	\label{concrete:tab}}
\end{table}

	\subsection{Breakdown of MDS in high-dimensional ambient noise}

	In the six decades since MDS was proposed, typical datasets  have grown in
	both size and dimension. MDS, as well as more recently proposed manifold
	learning techniques, are being applied to data
	of increasingly high ambient dimensions, departing from the setup 
	$ p \ll n$ 
	for which
	they were originally designed. In particular, when the data is
	high-dimensional, certain mathematical phenomena
	kick in, which fundamentally alter the behavior of MDS.

	In practice, even though the data is measured in an ambient
	high dimensional space,
	it often resides on a low-dimensional structure embedded in that space. Manifold
	learning techniques, for example, assume that the data resides on a
	low-dimensional smooth manifold embedded in the high-dimensional space. 
	For simplicity, consider an
 ``original'' dataset $\{\V{\Vx_1},\ldots,\V{\Vx_n}\}\subset \R^d$ that
resides on a $d$-dimensional
linear subspace ($d\ll p$), embedded in the ambient space $\R^p$, in which the data is
actually measured. It is natural to assume that the measurements are noisy, so
that we actually observe samples $\Vy_1,\ldots,\Vy_n$ with
\[
	\Vy_i = \rho(\V{\Vx_i})+\varepsilon_i \qquad i=1,\ldots,n \,,
\]
where $\rho:\R^d\hookrightarrow\R^p$ is the isometry embedding the low-dimensional
subspace into the ambient space, and where $\varepsilon_i\in\R^p$ are ambient noise vectors
($i=1,\ldots, n$). 

When the ambient dimension $p$ is not much smaller than the sample size $n$,
the presence of ambient noise can have drastic effects on the diagonalization 
step of MDS. In particular, in the related scenario of covariance matrix
estimation, results from high dimensional statistics \cite{Paul2007}
and random matrix theory \cite{Baik2005}
 have shown in that  the eigenvalues and eigenvectors of the data matrix
 deviate, sometimes significantly,
 from the embedding vectors they are presumed to estimate. As a result, the
 quality of the MDS embedding becomes sensitive to the ambient noise level.
 In this paper we demonstrate that a phenomenon names after Baik, Ben-Arous and
 P\'ech\'e \cite{Baik2005} occurs in MDS, whereby there is a sharp phase
 transition in the embedding quality (see examples in Section
	 \ref{examples:subsec} below,
 Figure \ref{fig:mnist} and Figure \ref{fig:helix}).

 Our second main contribution is 
 formal characterization of this phenomenon in MDS. We show 
 that, as the ambient noise level increases, 
 MDS suffers a sharp {\em breakdown},  
 and provide an asymptotically exact formula for
	the signal-to-noise level at which breakdown occurs.
	
	\subsection{An optimal variant of MDS}
	\label{SS:1OptVariant}
	Even before breakdown occurs, in the high-dimensional setting, 
	the quality of the MDS embedding deteriorates
	as the ambient noise level increases. This calls for an improvement of MDS,
	which is able to correct for the noise effects. Experience tells us that  
	complicated alternatives of MDS do not become widely used by scientists. 
	Instead, a simple variation on MDS is preferred, which can be 
	calculated easily based on the existing MDS methodology.

	A simple solution is available in the
	form of {\em eigenvalue shrinkage}. Recently, in the related problem of
	covariance matrix estimation, \cite{Donohoa2013}
	have shown that, by applying a carefully designed shrinkage function to the
	eigenvalues of a sample covariance matrix, it is possible to
	significantly mitigate the effects of high-dimensional ambient noise.
	
	Here, we consider a simple variant of MDS that applies a univariate 
	shrinkage function $\eta:\R\to\R$ to the eigenvalues $d_1,\ldots,d_n$ of the
	MDS matrix $S$ from \eqref{eq:S}. Instead of using the eigenvalues 
	$d_1,\ldots, d_n$ of $S$ as in \eqref{eq:Sdiag} above, we 
	use the shrunken values $\eta(\sqrt{d_1}),\ldots,\eta(\sqrt{d_n})$ 
 -- see example in Section
	 \ref{examples:subsec} below,
	 Figure \ref{fig:helix_opt1} and Figure \ref{fig:helix_opt2}. In fact, 
	 classical MDS
	 turns out to be equivalent to a specific choice of hard threshold shrinker.
	 The question naturally arises whether better shrinkers can be designed,
	 which outperform classical MDS.

	The third main contribution 
	of this paper is
	\MDSp, formally defined in Table \ref{MDS:alg} below.
	\MDSp~ is a simple variant of MDS which applies a carefully derived shrinkage
	function to the eigenvalues $d_1,\ldots,d_n$
	before proceeding with the MDS embedding.

	Concretely, \MDSp~ is a simple modification of MDS: In step 4 above,
	instead of the MDS embedding that uses $\sqrt{d_i}\cdot \V{u}_i$ we embed
	using $\eta^*(\sqrt{d_i})\cdot \V{u}_i$, where $\eta^*$ is the 
	{\em optimal shrinker} for MDS:
\begin{eqnarray*}
\eta^* (\Shy) =
\left\{ \begin{array}{ll}
\sigma \sqrt{(\Sx(\Shy)/\sigma)^2-\beta-\frac{\beta\cdot
		(1-\beta)}{(\Sx(\Shy)/ \sigma  )^2+\beta}} & 
\,\, \Shy>\sigma\cdot(1+\sqrt{\beta})\\
0 & \,\,otherwise
\end{array} \right.
\end{eqnarray*}
where 
\begin{eqnarray*}
\Sx(\Shy)=\frac{\sigma}{\sqrt{2}} \sqrt{ \left(\frac{\Shy}{\sigma}\right)^2-1-\beta+\sqrt{\left(\left(\frac{\Shy}{\sigma}\right)^2-1-\beta\right)^2-
		4\beta}}
\,.
\end{eqnarray*}
Here, as above, $\beta=(n-1)/p$ and $\sigma$ is the standard deviation of the
noise, replaced by the consistent robust estimate $\hat{\sigma}$ from \eqref{eq:SigmaEstimation}
 in case $\sigma$ is unknown.

	Figure \ref{fig1} below compares the optimal shrinker underlying \MDSp~ with the hard
	threshold shrinker underlying classical MDS.
	\begin{figure}[H]
		\includegraphics[width=1\linewidth]{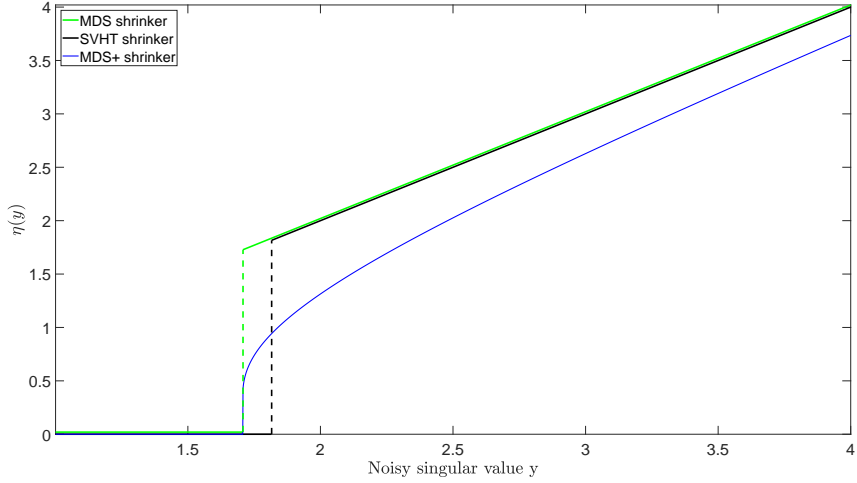}
		\caption{The optimal shrinker derived in this paper, compared
		    with an optimal hard threshold SVHT (see below) and the classical
		    MDS shrinker. Specific parameters are used --
		   see figure
		   \ref{fig:shrinkers_funcs} and section \ref{S:3}. (Color online)
		\label{fig1}
		}	
	\end{figure}
	As the figure shows, if $\sqrt{d_i}$ is too small, and specifically if
	$d_i\leq \sigma\cdot(1+\sqrt{\beta})$, 
	$\eta^*(\sqrt{d_i})=0$. In this case the vector $\V{u}_i$ is not used
	in the embedding, resulting in a smaller embedding dimension. In other
	words, the embedding dimension used by \MDSp~ equals the amount of
	eigenvalues of the matrix $S$ that fall above the value 
	$ \sigma^2\cdot(1+\sqrt{\beta})^2$. Table \ref{MDS:alg} below summarizes
	the \MDSp~ algorithm.

	Measuring the embedding
	quality using a natural loss function, we prove that \MDSp~ offers
	the {\em best possible embedding}, asymptotically, among any eigenvalue 
	shrinkage variant of MDS.

	\begin{table}[H]
		\centering
		\begin{tabular}{l}
			\hline
			\textbf{\MDSp}\\
			\hline
			Input: distance matrix $\Delta$ and dimensionality $p$.
			\\
			Let $\beta=(n-1)/{p}$. Let the value of $\sigma$ be given or 
			estimated (see  Theorem \ref{thmSigmaEstimation} below).
			\\
			1. Create the similarity matrix\\
			$\hspace{37mm} S=-\frac{1}{2}H\cdot \Delta \cdot H$\\
			2. Diagonlize S to obtain\\
			$\hspace{40mm} S=U\cdot D \cdot U'$\\
			where $U\in O(n)$ and $D=diag(d_1,\ldots,d_n)$ s.t.
			$d_1\geq\ldots\geq d_n$.\\
			3. Estimate the embedding dimension\\
			$ \hspace{20mm} r= \#\left\{i\in \{1,\ldots,n\}:{\ }d_i>\sigma
			\cdot (1+\sqrt{\beta}) \right\} $.\\
			4. Return a $n$-by-$r$ 
			matrix with columns $\eta^*(\sqrt{d_i})\cdot\V{u}_i$ ($i=1,\ldots,r$),\\ 
			where $\eta^*$ is as in Theorem \ref{thmmdsOptShrinkerExplicit}.
			Embed the points into $\R^r$ using the rows of the matrix.
			\\
			\hline
		\end{tabular}
		\caption{The \MDSp~ Algorithm \label{MDS:alg}}
	\end{table}

	\subsection{Examples} \label{examples:subsec}
	As a gentle introduction to our results, 
	we consider two simple examples of MDS from noisy, high-dimensional data.
\\~\\
	{\bf MNIST.} 
	The famous MNIST dataset \cite{LeCunYann1998} contains greyscale images of hand-written
	digits. Clustering after MDS
	embedding (a form of spectral clustering) is often used to distinguish between different digits. 
	For illustration purposes, we studied 700 images of the digits $0$ and $1$
	(Figure \ref{fig:zero-one}), with varying levels of added white Gaussian
	noise. As the data consists of two distinct clusters, it should be enough to
	use MDS embedding into $r=2$ dimensions. 
	Figure \ref{fig:mnist} clearly shows the deteriorating embedding quality as
	the noise level increases, and the eventual MDS breakdown, accurately 
	predicted by
	Theorem \ref{thmThmMDSExplicitLoss} below.

	\begin{figure}[h!]
		\centering		\includegraphics[width=0.45\linewidth]{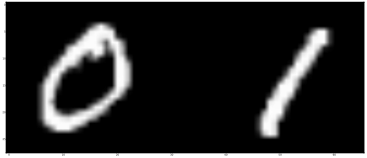}
		\caption{Example of images from MNIST \cite{LeCunYann1998}
 }
		\label{fig:zero-one}
	\end{figure}

	\begin{figure}[h!]
	\includegraphics[trim={4cm 0 0 0},clip, width=1\linewidth]{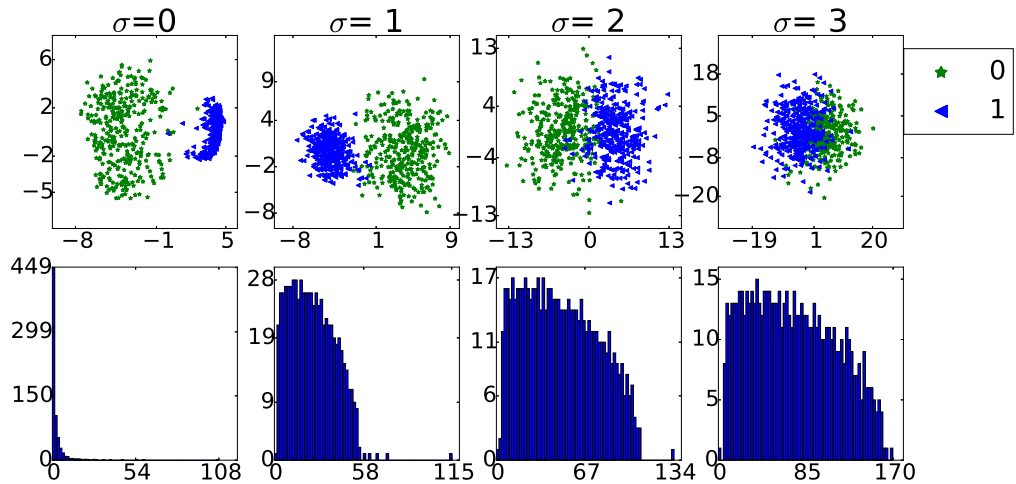}
	\caption{Breakdown of MDS on MNIST data. $n=700$ images in $p=784$ dimensions
		(pixels). Gaussian i.i.d noise $\Nc(0,\sigma^2)$ added as indicated in
		each column.
		Top panels: MDS embedding into $r=2$ dimensions.
		Bottom panels: histogram of the corresponding 
		square root of the spectra of the MDS matrix $S$ from 
		\eqref{eq:S}. Observe that breakdown occurs exactly as the noise-related 
	eigenvalues, which grow proportionally to $\sigma$, ``engulf'' the lowest
structure-related eigenvalue.   (Color online)}
\label{fig:mnist}
	\end{figure}

~\\
		{\bf Helix.} 
		A recent application of MDS in molecular biology is analysis of 
		Hi-C measurements \cite{tanizawa2010HiCHumanPaper}. Here,
		MDS is used to recover the three-dimensional
		packing of DNA molecules inside the cell nucleus from measurements of
		spatial affinity between loci along the genome. As a toy example for
		this reconstruction problem, we
		consider a reconstruction of a helix-shaped point cloud in $\R^3$ 
		from the pairwise
		distances between $n=300$ points in the cloud. 
		The point cloud was embedded in a high dimensional space $\R^{500}$, and
		i.i.d Gaussian ambient noise of a varying level was added. 
		Figure \ref{fig:helix} demonstrates the deterioration of MDS embedding quality
		with increasing noise level,
		and the corresponding spectra of the MDS matrix. 
		
		Figure \ref{fig:helix_opt1} demonstrates
		the effect of the improved MDS algorithm we propose, based on
		optimal shrinkage of the MDS eigenvalues. The
		breakdown of MDS is apparent in the right panel.
		The optimal shrinker identifies that one of
		the MDS axes is non-informative, and shrinks the corresponding
		eigenvalue to zero. As a result, the \MDSp~ embedding is
		two-dimensional (Figure \ref{fig:helix_opt2}).

	\begin{figure}[h!]
		\centering
		\includegraphics[trim={5cm 0 0 0},clip,width=1.2\linewidth]{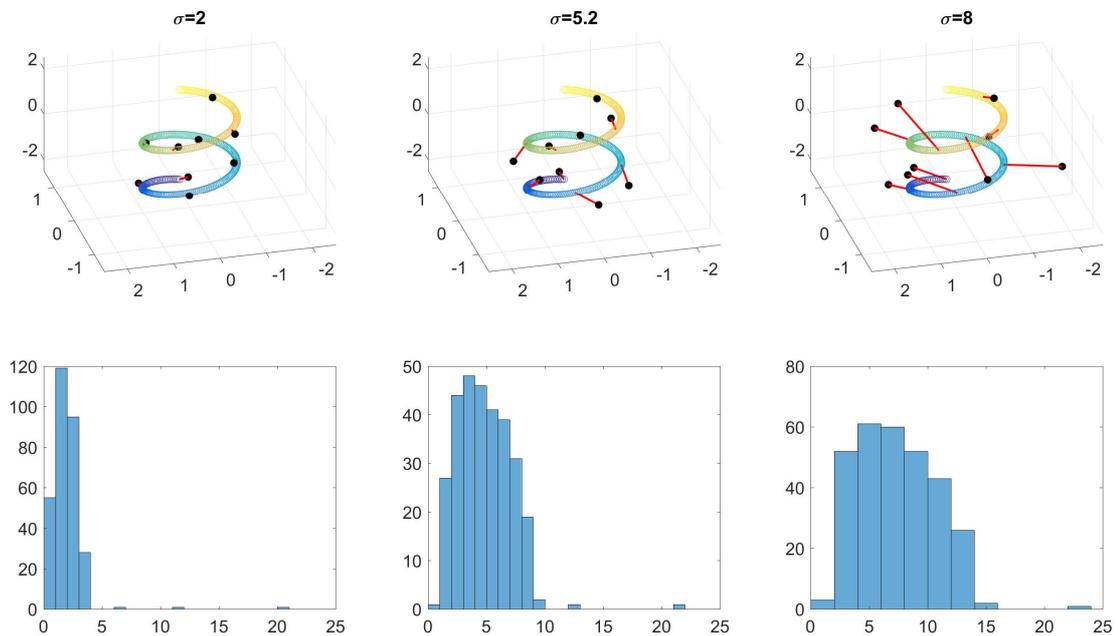}
	\caption{ Breakdown of MDS on a helix simulation. $n=300$ points along a
		helix embedded in $p=500$ dimensions.
		Gaussian i.i.d noise $\Nc(0,\sigma^2)$ was added as indicated in
		each column.
		Top panels: MDS embedding into $r=3$ dimensions (solid: original helix
		shape. black markers: some points of the embedded data. red lines: displacement vectors between the embedding and real position of the data).
		Bottom panels: histogram of the corresponding 
		square root of the spectra of the MDS matrix $S$ from 
		\eqref{eq:S}. Observe that breakdown occurs exactly as the noise-related 
	eigenvalues, which grow proportionally to $\sigma$, ``engulf'' the lowest
structure-related eigenvalue. (Color online)}
		\label{fig:helix}
	\end{figure}

	\begin{figure}[h!]
		\centering
		\includegraphics[trim={5cm 0 0
		0},clip,width=1.\linewidth]{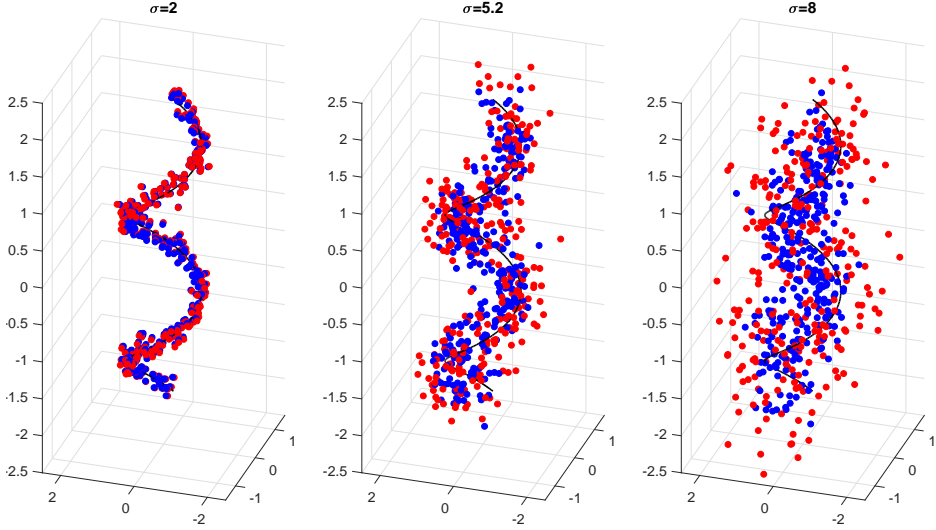}
	\caption{The \MDSp~ algorithm on the helix simulation from Figure
		\ref{fig:helix}.
Red: classical MDS. Blue: \MDSp. An additional view 
of the right panel is available in Figure \ref{fig:helix_opt2} below.
(Color online)}
		\label{fig:helix_opt1}
	\end{figure}

\begin{figure}[h]
		\centering
		\includegraphics[trim={5cm 0 0
		0},clip,width=1.0\linewidth]{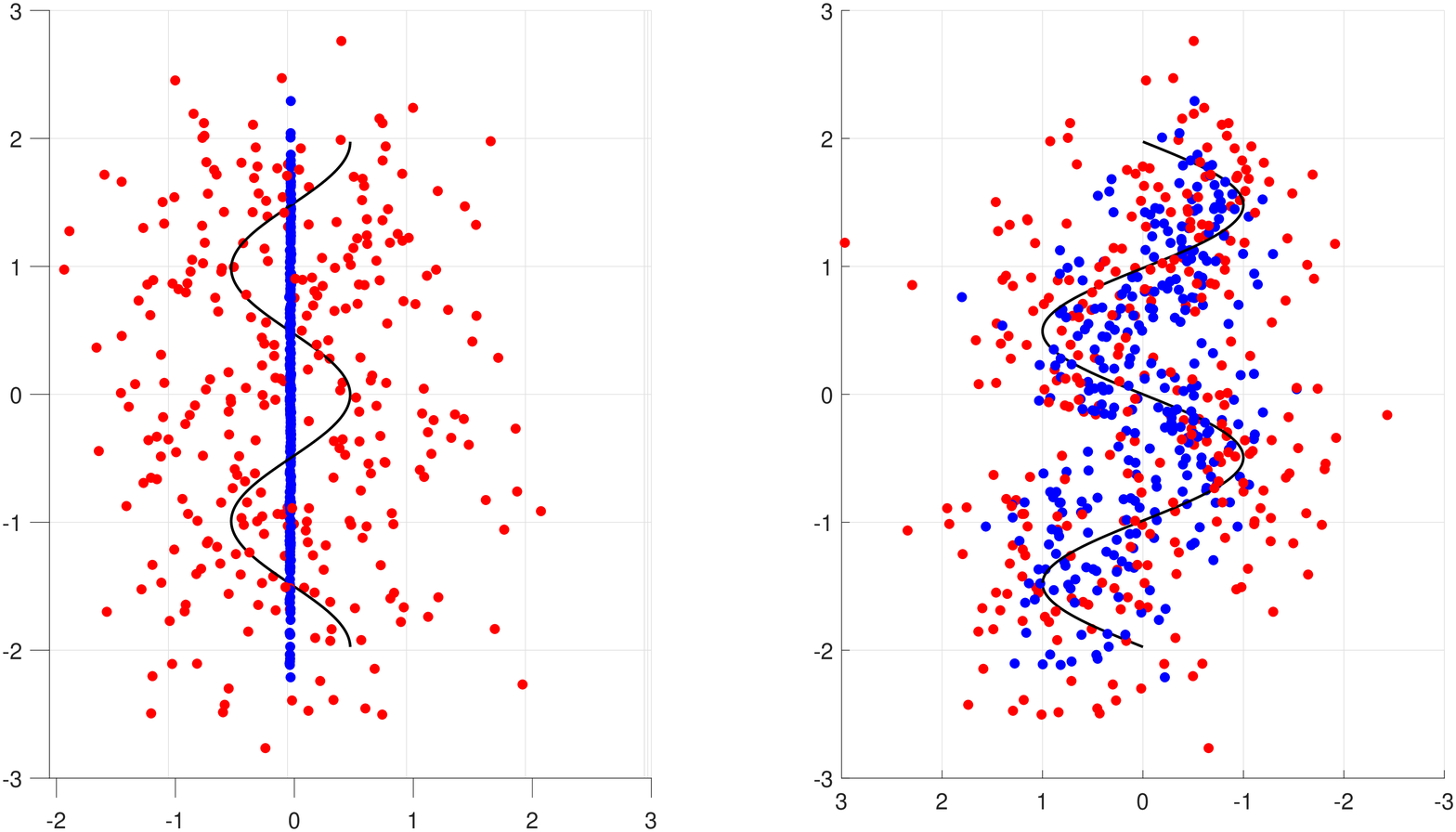}
	\caption{The \MDSp~ algorithm on the helix simulation from Figure
		\ref{fig:helix} in high noise, $\sigma=8$.
		Red: classical MDS. Blue: \MDSp. Embedding into $\R^3$ spanned by the
		standard basis vectors $e_1,e_2,e_3$. Left: the $(e_2,e_3)$ plane.
		Right: the $(e_1,e_3)$ plane.
(Color online)}
		\label{fig:helix_opt2}
	\end{figure}

	\subsection{Related work} 
~	
	\paragraph{Classical MDS}  
    Precursors of MDS go back as far as 1938 to Young and Householder's work on
	Classical Scaling algorithm \cite{young1938discussion}.
	The MDS algorithm as we know it today has been developed by several authors,
	notably 
	Togersson \cite{torgerson1958theory} and Gower \cite{gower1966some}.

	\paragraph{Choosing the embedding dimension}
	
	Classical choice of embedding dimension is based on the fact
	that the similarity matrix \eqref{eq:S} should be
	positive semidefinite.  
	
	The literature offers two popular heuristics for choosing the embedding
	dimension that are based on this fact. In the {\em Scree Plot} method
	\cite{Cattell1966}  one plots 
	the eigenvalues $(d_1,\ldots,d_n)$ from \eqref{eq:Sdiag}  in decreasing order over
	$(1,\ldots,n)$,
	and looks for the ``inflection'', ``knee'' or ``elbow'' in the plot to
	determine $r$. 
	The rational behind this visual heuristic is that the ``scree'', namely the 
	slowly changing eigenvalues (below the inflection point) are due to the
	noise,
	while the top eigenvalues (above the inflection point) are due to the
	$d$-dimensional subspace containing the signal.
	In a different heuristic for choosing $r$, one seeks to 
	maximize the function
		\begin{equation}
	\frac{\sum_{i=1}^{r}d_i}{\sum_{i=1}^{rank(S)} d_i}
	\end{equation}
	while keeping the selected embedding
	dimension $r$ as low as possible \cite{cox2000multidimensionalPage38}. 
	
	In some cases the actual distance matrix gets corrupted, and as a result the
	similarity matrix may have one of more negative eigenvalues. In this case it
	is commonly suggested to choose the embedding dimension by maximizing one of
	the following target functions \cite{cox2000multidimensional}:
	\begin{equation}
		\frac{\sum_{i=1}^{r}\lambda_i}{\sum_{i=1}^{rank(S)} |\lambda_i|} 
		\qquad \text{or} \qquad
		\frac{\sum_{i=1}^{r}\lambda_i}{\sum_{i=1}^{rank(S)} max(0,\lambda_i)}
	\end{equation}
	
	In another heuristic for choosing $r$, due to
	Kruskal \cite{kruskal1964multidimensional}, one 
	considers the so-called {\em Stress-1} function: 
	\begin{equation}
	\sigma_1=	\sqrt{\frac{\sum_{i,j} (\Delta_{i,j} -\hat{\Delta}_{i,j})^2}{\sum_{i,j} (\Delta_{i,j})^2}}
	\end{equation}
	where $\hat{\Delta}_{i,j}$ is the euclidean distance between the embedding
	of point $i$ and the embedding of point $j$. 
	Kruskal proposed to choose the embedding dimension $r$ by minimizing
	$\sigma_1$. 

	One striking observation regarding the selection of embedding dimension is
	that even though MDS and its variants are extremely popular, the literature
	does not propose a systematic method that is backed up by rigorous theory. 
	
	\paragraph{Algorithm performance in the presence of noise}
	To the best of our knowledge, the literature does not offer a systematic
	treatment on 
	the influence of ambient noise on MDS embedding quality.
	Kruskal's variation of MDS  \cite{Kruskal1964} was studied by Cox and Cox
	\cite{cox1990interpreting}, 
	who have shown by simulation that the stress function 
	is an almost perfectly linear function of the noise when $d=r=2$,
	independently of the amount of samples. 

	\subsection{Outline}
	This paper proceeds as follows. In Section \ref{S:2} we provide the formal
	problem setup, and propose a loss function to quantify the quality of any 
	low-dimensional embedding algorithm.
	Our main results are stated in Section \ref{S:3} and discussed in 
	 Section \ref{S:5}.
	The proofs appear in Section \ref{S:6}, where we also show how to estimate
	the ambient noise level.
While our main results are stated under the assumption $n-1\leq  p$, 
in the Appendix we rigorously
extend our results to the case $n-1\geq p$.

\subsection{Reproducibility Advisory}

The algorithms developed in this paper have been implemented in both 
{\tt Python} and  {\tt Matlab} and are made available in the 
code supplement \cite{CODE}. The code supplement also includes source code to
reproduce all the figures in this paper.

\section{Problem setup}
\label{S:2}	

\subsection{Notation}
\label{Notation}
Let $a_+= max(a,0)$ denote the positive part of $a\in \R$.
We use boldface letters such as $\V{\Vy}\in\R^p$ to denote a vector with coordinates
$[(\V{\Vy})_1,\ldots,(\V{\Vy})_p]$ and Euclidean norm $\norm{\V{\Vy}}_2$. 
We use capital letters such as
$A$ to denote a matrix with transpose $A^\Tr$. The $i$-th column and $j$-th row 
of $A$ will be
denoted by $A_{*,i}$ and $A_{j,*}$ respectively.
Denote the set of $m$-by-$n$ real matrices by $\M_{m\times n} $ and
the set of orthogonal matrices by $O(n) \subset \M_{n\times n}$.
The Frobenius norm of $A\in\M_{m\times n}$ is defined by
\[
	\norm{A}_F = \sqrt{\sum_{i=1}^m\sum_{j=1}^n A_{i,j}^2}\,.
\]
Let $\V{1}_p\in \R^p$ denote the vector $\V{1}_p=(1,\ldots,1)$
and let $1_{m\times n}$ denote the ``all-ones'' $m$-by-$n$ 
matrix $1_{m \times n} =
\mathbf{1}_m \mathbf{1}_n^\Tr$. Similarly, 
let $\mathbf{0}_p\in \R^p$ denote the vector $\mathbf{0}_p=(0,\ldots,0)$
and let
$0_{m \times n}$ denote the
``all-zeros'' $m$-by-$n$ matrix.
 For some $d_1,\ldots,d_n$ we denote by $diag(d_1,\ldots,d_n)$ the
 $n$-by-$n$ diagonal matrix with main diagonal $d_1,\ldots, d_n$. 
 The $n$-by-$n$ identity matrix $diag(1,\ldots,1)$ is denoted by
 $I_n$.
 We also denote the $m$-by-$n$  "all-zeros" matrix, with ones only on its main diagonal using $I_{m\times n}$.
Finally, we use $H$ to denote the
MDS centralization matrix 
\begin{equation}
\label{eq:centralizationMat}
H=I_n-\frac{1}{n}\V{1}\cdot \V{1}^\Tr= I_n - \frac{1}{n} 1_{n\times n}\,.
\end{equation}

\subsection{Setup}
\label{Setup}
In this paper we consider MDS and its variants when applied to 
noisy, high dimensional measurements of a dataset with low intrinsic dimension.
Let $d$ denote the (low) intrinsic dimension and assume that we are interested
in the unknown, unobservable dataset $\{	\V{\Vx_i}\}_{i=1}^n\subset \R^d$.
These data are embedded in a high dimensional space $\R^p$ via an unknown rotation
matrix $R\in O(p)$, such that 
\begin{eqnarray}
\label{eq:aTilde}
	\V{\Vxt_i} = R \cdot\left( 
		\begin{array}[h]{c}
			\V{\Vx_i} \\ \V{0_{p-d}}
		\end{array}
	\right)
\qquad i=1,\ldots,n
\,.
\end{eqnarray}

However, we only observe a noisy version
of the embedded dataset, which we denote by $\{\V{\Vy_i}\}_{i=1}^n\subset \R^p$.
Formally,
\begin{eqnarray}
\label{eq:notationDataModel}
	\V{\Vy_i} = \V{\Vxt_i} + \V{\varepsilon_i} \qquad i=1,\ldots,n
\end{eqnarray}
where
 $\{\V{\epsilon_i}\}_{i=1}^n \overset{iid}{\sim}\Nc_p(\V{0}_p,\frac{\sigma}{\sqrt{p}}\cdot I_p)$ 
 is ambient noise of level $\sigma$. (Note that in the introduction we used
     noise level without normalization; henceforth noise level $\sigma$ implies
 noise standard deviation $\sigma/\sqrt{p}$.) 
While we assume $d< n \leq p$, we will focus on the regime 
$d\ll n$ and $n\sim p$. The assumption $n-1 \leq p$ helps simplify results and
proofs; the
case $n-1\geq p$ is discussed in the Appendix. 

It is convenient to stack the data vectors as rows and 
form the data matrices 
$\Mx\in \R^{n \times d}$ and $\My\in \R^{n \times p}$
such that 
\begin{eqnarray}
\label{eq:datasetMatX}
\Mx_{i,*}=\V{\Vx_i}^{\T} \qquad i=1,\ldots,n
\end{eqnarray}
\begin{eqnarray}
\label{eq:datasetMatY}
\My_{i,*}=\V{\Vy_i}^{\T} \qquad i=1,\ldots,n \,.
\end{eqnarray}
Observe that both $X$ (resp. $Y$) is a multivariate data matrix with $n$ rows,
or samples, and $d$ (resp. $p$) columns, or features.
we define the aspect ratio of the matrix $\My$ as $\beta= {(n-1)}/{p}$.
\\

For simplicity, we will assume that the data $\{\V{\Vx_i}\}_{i=1}^n$ is centered around the origin, meaning that $H\cdot X=X$. 
Denote the singular values of the $n$-by-$d$ matrix $\Mx$  by
$x_1\geq \ldots \geq x_d \geq 0 $ and the singular values of the $n$-by-$p$ matrix $H\cdot
\My$ by
$y_1\geq \ldots \geq y_n \geq 0$.

\subsection{The classical MDS algorithm}
\label{Classical MDS}
The classical MDS algorithm, described briefly in the introduction,
is provided with two arguments. 
The first is $\Delta\in \M_{n\times n}$, a matrix that contains the pairwise
distances over the observable data, $\Delta_{i,j}=\norm{ \Vy_i -\Vy_j
}_{F}^2$.
The second is $r$, the dimension into which the data is to be embedded.
An equivalent formal description of classical MDS consists of the two following steps:
\begin{enumerate}
	\item Define the similarity matrix
	\begin{equation}
	\label{eq:S2}
	S= -\frac{1}{2} H\cdot \Delta \cdot H
	\end{equation}
\item Find a $n$-by-$r$ matrix $\hat{X}$ by  
	\begin{equation}
	\label{eq:Xhat}
		\Mxh^{MDS}=
	\underset{Z\in \M_{n\times r}: {\ } \tilde{V}_Z=I}{argmin} \parallel S- Z\cdot Z^{\T} \parallel_{F}^2
	\,,
	\end{equation}
	where $\tilde{V}_Z$ is the right singular vector matrix of $Z$, and use the
	rows of $\Mxh$ to embed the $n$ data points in $\R^r$. In other words, we
	optimize over matrices $Z\in \M_{n\times r}$ such that $Z^\T Z$ is diagonal.
	Note that 
 $\Mxh_{i,*}$ is the embedding coordinates of $i$-th datapoint. While
 $\Mxh$ depends on $r$, we leave this dependency implicit in the notation. Being
 the number of columns of $\Mxh$, it is easy to infer from context.\\
\end{enumerate} 
 
It is easy to verify that the MDS algorithm  mentioned in section \ref{S:1} is equivalent to the
one above. In fact, the MDS admits a more convenient formulation, as follows.
 Theorem \ref{thmMDSRealLoss} below states that
 $S= (HY) \cdot (H Y)^T$. As a result, we have the following lemma.

 \begin{lem} Let $S\in M_{n\times n}$ be a similarity matrix as in
	 \eqref{eq:S2}.  Then for any $i = 1,\ldots, rank(S)$ we have
	 \begin{enumerate} \item $\sqrt{d_i}=\Shy_i$, where $d_i$
			 is the $i$-th eigenvalue of $S$.  
		     \item There exists $q\in \{\pm 1\}$ such that
			 $$ \V{u}_i=q\cdot \V{w}_i $$ where $\V{u}_i$ is the i-th left
			 singular value of $H\cdot Y$ and $\V{w}_i$ is the i-th eigenvector
			 of S.
 	
 	\end{enumerate}
 \label{lemma:SandHY_SVD}
 	 \end{lem} 

 It follows that the MDS embedding \eqref{eq:Xhat} is given equivalently by
	 \begin{equation}
	     \label{Xhat_MDS:eq}
 	\Mxh^{MDS}= \sum_{i=1}^r \Shy_i q_i ~ \V{u_i} \V{e_i}^{\T}
 \end{equation} 
 where $e_i\in \R^r$ are the standard basis vectors and
 $q_1,\ldots,q_r\in\left\{ \pm 1 \right\}$.
 
\subsection{Formal analysis of MDS accuracy}
\label{FormalAnalysisMDSAccuracy}
MDS was originally developed for the noiseless scenario. Indeed, when no noise
is present, if the parameter $r$ provided equals to the latent dimension $d$, 
it is well known that 
\begin{equation}
\Mxh_{i,*} ^ \T= 
R \cdot \V{\Vx_i} \quad i=1,\ldots,n
\end{equation}
for some $R\in O(d)$. In other words, in the absence of noise, if $r=d$,
MDS recovers the latent (low-dimensional) data vectors $\V{\Vx_1},\ldots,\V{\Vx_n}$ exactly - up to a
global rotation. A proof of this fact is provided in Section \ref{S:6}, 
see Theorem \ref{thmNoiselessMDSPerfectRecon}. 

Clearly, in the presence of noise ($\sigma>0$) one cannot hope for exact
recovery, and some formal measure of ``MDS accuracy'' is required. Such a notion
of accuracy is 
traditionally obtained by introducing a loss function. Consider the following
loss function, which measures the ``proximity'' of the point cloud recovered by
MDS to the original, unknown point cloud that MDS aims to recover.

\begin{defin}[Similarity distance]
    \label{def1}
	Given two datasets $\{\V{a_i}\}_{i=1}^n\subset \R^d$ , $\{\V{b_i}\}_{i=1}^n\subset \R^r$, where $d,r\in \N$, define a generic distance between the two datasets by
	$$
	M_{n}\Big(\{\V{a_i}\}_{i=1}^n,\{\V{b_i}\}_{i=1}^n \Big)=   \underset{R\in O(l)}{min}  \sum_{i=1}^n 
	\bigg| \bigg|
	\left[ \begin{array}{c} \V{a_i}-\frac{1}{n} \sum_{i=1}^{n} \V{a_i} \\\V{0_{l-d}} \end{array} \right] - 
	R\cdot 
	\left[ \begin{array}{c}
	\V{b_i} -\frac{1}{n} \sum_{i=1}^{n} \V{b_i}
	\\\V{0_{l-d}} \end{array}   \right]  \bigg| \bigg|_2
	$$
	where $l=max(d,r)$.\\
	\end{defin}
\noindent Observe that using our matrix notation
this formula could be written as
	\[
	    M_n\Big(\{\V{a_i}\}_{i=1}^n,\{\V{b_i}\}_{i=1}^n \Big)= 
	\underset{R\in O(l)}{min}  
	\bigg| \bigg|
	H\cdot \left[A, 0_n\times (l-d)^+  \right] - 
	H\cdot \left[B, 0_n\times (l-d)^+  \right]\cdot R 
	\bigg| \bigg|_2\,.
	\]
	Here, $A\in \M_{n\times d}$ is a column
	stacking of the $\{a_i\}$, namely  $A_{i,*}=a_i^{\T}$. Similarly
$B\in\M_{n\times r}$ is a column stacking of $\{b_i\}$.

While the function $M_n$ depends on $(n,d,r)$, we suppress $d$ and $r$ in the
notation $M_n$ and leave them to be inferred from context. 
We first observe that in the noiseless case, classical MDS does indeed 
find the
minimum of this loss function:

\begin{lem}
	\label{lemNoiselessConfigMDSGetOptimalSimilarityMat}
	Let $\sigma=0$, let $n\in\N$ be arbitrary and let $d<n$.  
	Then for any $r\geq d$, the MDS solution from \eqref{eq:Xhat}
	with embedding dimension $r$ satisfies 
	\begin{equation}
	\Mxh^{MDS} \in  \underset{A\in M_{n\times r}}{argmin} {\ } M_n (A,\Mx)
	\end{equation}
	Moreover, $M_n(\Mxh^{MDS},\Mx)=0$.
	
\end{lem}

It follows that an algorithm that tries to minimize $M_n(\cdot ,X)$ will agree with classical
MDS in the noiseless case.
We now argue that $M_n$ is a natural loss for measuring the MDS accuracy in the noisy
case as well. Observe that a reasonable loss for measuring MDS accuracy must satisfy the
following properties: 
\begin{enumerate}
	\item {\em Rotation invariance.} 
	We say that a loss $M_n$ is rotation-invariant if
	\begin{equation}
	\label{eq:RotationInvariantDef}
	M_n\left( Y_1, Y_2 \cdot R \right)= M_n\left( Y_1, Y_2  \right)
	\end{equation}
	for any two data matrices
	$Y_1,Y_2\in M_{n\times p}$ and
	any rotation matrix $R\in O(p)$.

	\item {\em Translation invariance.} 
	We say that  a loss $M_n$ is translation-invariant if
	\begin{equation}
	\label{eq:TranslationInvariantDef}
	M_n \left( Y_1,Y_2+ \left[ \begin{array}{c}
	\V{c}^T\\ 	\V{c}^T \\ \ldots \\ 	\V{c}^T
	\end{array} \right] \right) = 
	M_n \left( Y_1,Y_2 \right)
	\end{equation}
	for any two data matrices 
	$Y_1,Y_2\in M_{n\times p}$ and any translation vector $\V{c}\in \R^p$.
	
    \item {\em Padding invariance.} We say that a loss $M_n$ is
	padding-invariant if 
    	\begin{equation}
	\label{eq:ConstPaddingInvariantDef}
	M_n\left( \left[Y_1, \begin{array}{c}
	c^T \\ c^T \\ \ldots \\ c^T
	\end{array} \right], Y_2 \right) = M_n \left( Y_1 , Y_2 \right)
	\end{equation}
for any
$k\in \N$, $Y_1 \in M_{n\times p}$, $Y_2 \in M_{n\times (p+k)}$ and
$\V{c}\in \R^k$. 

\end{enumerate}

\noindent We now show that the loss $M_n$ from Definition \ref{def1} 
satisfies all these invariance properties.
\begin{lem}
	\label{lemSimlarityDistanceFromAxioms}
	Let $n,d$ and $r$ be such that $n>r$ and $n>d$. Then the function  $M_n: M_{n\times r}\times
	M_{n\times d} \to \R$ satisfies properties (1)-(3) above. 
\end{lem}

\noindent In fact,  $M_n$ turns out to be a pseudo-metric on
$M_{n\times r}\times M_{n\times d}$ when $r=d$:
\begin{lem}
    \label{pseudo_metric:lem}
    Let $n\in \mathbb{N}$ and assume $r=d$.
    The similarity distance $M_n$ from Definition \ref{def1} satisfies the following
    properties:
    \begin{enumerate}
	\item $M_n(Y,Y)=0$
	\item $M_n(X,Y) = M_n(Y,X)$ 
	\item $M_n(X,Y)\leq M_n(X,Y)+M_n(Y,Z)$ 
    \end{enumerate}
\end{lem}

\noindent In summary, we arrived at the following natural definition for a loss function measuring the
accuracy of MDS and MDS-type algorithms, which satisfied the fundamental
properties of an embedding accuracy loss function.
\begin{defin}
	\label{similarityLoss}
	Let $\Mx$ be a dataset with $n$ points in $\R^d$ 
	as in \eqref{eq:datasetMatX} and let $\My$ be a noisy, high dimensional
	version as in \eqref{eq:datasetMatY}. Let $\Mxh$ be an embedding in $\R^r$ 
	constructed from $H\cdot \My$, where $r$ is an embedding dimension chosen by the
	scientist or the embedding algorithm used. We use the loss function
	\[
	L_n(\Mxh|X) = M_n^2(\Mxh,\Mx)
	\]
	to measure the accuracy of the embedding $\Mxh$.
\end{defin}

\subsection{Asymptotic model}
\label{AsymptoticModel}
We now find ourselves in a familiar decision-theoretic setup, and familiar
questions naturally arise:
How does one choose an embedding algorithm with favorable loss $L_n$? How does
the classical MDS algorithm compare, in terms of $L_n$, to alternative
algorithms? Is there an algorithm that achieves optimal $L_n$ in some
situations? 

Unfortunately, in general, analysis of the loss function $L_n$ is a difficult problem in the
presence of ambient noise due to the complicated joint distribution of the
singular values of $\My_n$ \cite{Forrester2010,DonohoGavish2013}. Recently, a line of works building on
Johnstone's Spiked Covariance Model \cite{johnstone2001distribution}, an asymptotic model
which considers instead a sequence of increasingly larger matrices, has yielded
exact forms of asymptotically optimal estimators, which were shown to be useful
even in relatively small matrices in practice
\cite{Benaych-Georges2012,Shabalin2013,Donoho2013_4_3,Gavish2017,donoho2013optimal}.\\

Following this successful approach, in this paper
we consider a sequence of increasingly larger embedding problems.
In the $n$-th problem, we observe $n$ vectors in dimension $p_n$, with $p$
growing proportionally to $n$. 
The original (low-dimensional) data matrix will be
denoted $\Mx_n\in M_{n\times d}$ (as in \eqref{eq:datasetMatX}), and the observed data matrix will
be denoted  $\My_n \in M_{n \times p_n}$
 (as in \eqref{eq:datasetMatY}).
Let $\My_n = [\Mx_n, 0_{n\times (p-d)}] \cdot R_n +Z_n$
where
$R_n\in O(p_n)$ and $Z_n,\My_n \in M_{n \times p_n}$  
satisfy the following properties: 

\begin{enumerate}
	\item \textit{Invariant white noise.} 
		The entries of $Z_n$ are i.i.d distributed, and drawn from a
		distribution with zero mean, variance $\sigma^2/p_n$, and finite
		fourth moment.		
		We assume that this distribution is orthogonally invariant in the sense that for any $A\in O(n)$ and $B\in O(p_n)$ the matrix $A\cdot Z_n\cdot B$ would follow the same distribution as $Z_n$. 
	
	\item \textit{Fixed signal column span($\Sx_1,\ldots, \Sx_d$).} 
		Let $d>0$, choose $\V{\Sx} \in \R^d$ with coordinates
		$\V{\Sx}=(\Sx_1,\ldots,\Sx_d)$ such that $\Sx_1 >  \ldots >\Sx_d > 0$. Assume for any $n\in \N$
	\begin{equation}
	\label{eq:ConfigurationX}
	\Mx_n= H\cdot \Mx_n =U_n \cdot [diag(\Sx_1,\ldots,\Sx_d), 0_{d\times (n-d) }]^{\T} \cdot \tilde{U_n}^{\T}
	\end{equation}
	is an arbitrary singular value decomposition of $\Mx_n$, where $U_n\in O(n)$ and $\tilde{U}_n\in O(d)$ are arbitrary and unknown 
	orthogonal matrices. Additionally, we preserve the unbiased assumption on the original data.
	
	\item \textit{Asymptotic aspect ratio} - Let $p_n$ be an increasing monotone
		sequence over $\N$, such that $\lim_{n\to\infty}
		{(n-1)}/{p_n}= \beta\in (0,1]$. We consider the case 
		$\beta\in[1,\infty)$ in the appendix section, under which we get the exact same results.

\end{enumerate}

Let $\Delta_n$ be the Euclidean distance matrix on the $n$-th problem, so that 
\[
    (\Delta_n)_{i,j} = \norm{\V{b}_{i,n}-\V{b}_{j,n}}^2
\]
where $\V{b}_{i,n}$ is the $i$-th row of $Y_n$.

Let $\Mxh$ denote an embedding algorithm, or more precisely a sequence of
embedding algorithm, one for each dimension $n$. In the $n$-th problem the
algorithm is given the input $\Delta_n$.
We abuse notation by using the symbol
$\Mxh$ to denote embedding regardless of data dimensions $n$ and $p_n$, so that
$\Mxh(\Delta_n)$ is the result of the embedding algorithm applied to the data $\Delta_n$.
Define the asymptotic loss of $\Mxh$ at $\mathbf{x}$ by
\[
	L(\Mxh|\mathbf{x}) \equiv \lim_{n\to\infty} 
	L_n(\Mxh(\Delta_n)|\Mx_n)
\]

assuming this limit exists. Following Lemma \ref{lem:MDSLossLemma} and Lemma \ref{lem:LossSeperateSigValsLemma}, it is easy that $L(\hat{X}|\textbf{x})$ is well defined when $\hat{X}$ is a shrinkage estimator as defined next. \\
In this asymptotic setting, the decision-theoretical problem becomes simple and
well-posed: Nature chooses the value $d$
and the vector $\V{\Sx}\in\R^d$, both unknown to the scientist. The scientist
chooses the embedding algorithm $\Mxh$, which includes a choice of the
embedding dimension $r$. After both ``players'' move, the ``payoff''
is the asymptotic loss   $L(\Mxh|\V{\Sx})$.

\subsection{Shrinkage} \label{Shrinkage} With the loss function, measuring
embedding accuracy, at hand,
the questions mentioned in the introduction become more
concrete: How does one design an embedding algorithm $\hat{X}$ with appealing
loss for a wide range of possible $\mathbf{x}$?  Is there an optimal choice of
$\hat{X}$ in some sense? As we will see, under the asymptotic loss $L$, both
these questions admit simple, appealing answers. \\

\noindent {\em Truncation estimators.} While the classical MDS estimator requires an estimate of the embedding
dimension $r$, we are interested in algorithms that do not require a-priori
knowledge or estimation of the embedding dimension. Let us define a ``padded''
version of the classical MDS algorithm from \eqref{Xhat_MDS:eq} by
\begin{eqnarray}
\label{eq:TSVDForm}
	\Mxh^{r}= \sum_{i=1}^{r} \Shy_i q_i {\ }\V{u_i}{\ } \V{e_i}^{\T}  
\end{eqnarray}
with $e_i\in \R^p$ and
$q_1,\ldots,q_r\in\left\{ \pm 1 \right\}$. Clearly, $\Mxh^{r}$ is just a zero-padded version of 
$\hat{X}^{MDS}$, in the sense that
\[
    \hat{X}^{r}(\Delta)= \left[ \hat{X}^{MDS}(\Delta), 0_{n\times (p-r)}
    \right]\,.
\]
The estimator $\Mxh^{r}$ acts by {\em truncating} the data eigenvalues $y_i$,
 keeping      only the $r$ largest ones; let us call it the Truncated SVD (TSVD)
 estimator. 
As our loss function is invariant under zero padding of the data matrix, it is
harmless to use $\Mxh^r$ instead of $\Mxh^{MDS}$. Below, we take the TSVD
estimator $\Mxh^r$ to represent the classical MDS.

In the introduction we mentioned the inherent conundrum involved in choosing the
embedding dimension $r$ for classical MDS. Suppose that there exists a function
$r^*$ mapping a pair $(\sigma,\Delta)$, where $\sigma$ is the noise level and
$\Delta$ is the observed distance matrix, to an ``optimal'' choice of embedding
dimension $r$ for classical MDS. Formally,
\begin{defin}[Optimal TSVD Estimator]
    \label{r_star:def}
	Let $\sigma>0$.
	Assume that there exists $\hat{r}^*: \R_{+}\times M_{n\times p} \rightarrow \N$
	such that 
	\[
	    L(\hat{X}^{\hat{r}^*}|\V{x})\leq L(\hat{X}^{r}|\V{x})
	\]
	for any $r \in \N$, $d\in N$ and $x\in \R^d$.
	Then $\hat{r}^*$ is called the optimal truncation value, and 
	$\hat{X}^{\hat{r}^*}$ is called the
	optimal TSVD estimator.
\end{defin}
\noindent We abuse notation by using $\hat{r}^*$ instead of
$\hat{r}^*(\sigma,\Delta)$.
Clearly, if such a function $\hat{r}^*$ exists, it would provide a definitive, 
disciplined manner
of choosing the embedding dimension $r$ for classical MDS. As we will see in the
next section, $\hat{r}^*$ does indeed exist and admits a simple closed form.
\\~\\
\noindent{\em Hard thresholding estimators.} The classical MDS estimator, in the form  $\Mxh^r$, is equivalent to a
different estimator, one which uses {\em hard thresholding} of the data singular
values. For $\lambda>0$, define 
\begin{eqnarray}
\label{eq:SVHTForm}
\Mxh^{\lambda}= \sum_{i=1}^{n} \Shy_i q_i\cdot 1_{[\Shy_i>\lambda]}  {\ }\V{u_i}{\ }
\V{e_i}^{\T}\,,
\end{eqnarray}
where $e_i \in \R^p$ and
$q_1,\ldots,q_{n}\in\left\{ \pm 1 \right\}$.
While $\Mxh^r$ keeps the $r$ largest 
data singular values, regardless of their size,  the estimator $\Mxh^{\lambda}$
keeps all the data singular values above the hard threshold $\lambda$. We call 
$\Mxh^{\lambda}$ a Singular Value Hard Threshold estimator, or {\em SVHT}.

It is easy to check that the family of estimators $\{\Mxh^r\}$ (with
data-dependent $r$) is in fact
equivalent to family of estimators $\{\Mxh^{\lambda}\}$ (with value of $\lambda$
fixed a-prior). Formally,
\begin{lem}
	\label{lemSVHT_TSVD}
	Let $Y\in M_{n\times p}$ be an observed data matrix, and $\Delta\in \M_{n\times n}$ be its corresponding euclidean distance matrix. For any TSVD estimator $\Mxh^{r}$, there exists a SVHT estimator $\Mxh^{\lambda}$ s.t. 
	\begin{eqnarray*}
	\Mxh^r (\Delta)= \Mxh^{\lambda}(\Delta)
	\end{eqnarray*}
	and vise versa. 
\end{lem}

\noindent How should one choose the hard threshold $\lambda$? 
Suppose that there exists an ``optimal'' value
$\lambda^*>0$ 
for which the asymptotic loss is always minimized. Formally,
\begin{defin} [Optimal SVHT Estimator]
    \label{lambda_star:def}
Let $\sigma>0$.
Assume that there is a value $\lambda^*\in[0,\infty)$, which depends on $\sigma$
and the asymptotic aspect ration $\beta$, 
    such that  
    	\[
	 L(\Mxh^{\lambda^*}|\V{x}) \leq L(\Mxh^{\lambda}|\V{x})
     \]
     for any 
     $ \lambda>0$, $ d\in \N$ and  $\V{x}\in \R^d$.
     Then $\lambda^*$ is called an optimal hard threshold, and 
     $\Mxh^{\lambda^*}$ is called the optimal SVHT. 
\end{defin}
\noindent For notational simplicity, write $\lambda^*$ instead of $\lambda^*_{\beta,\sigma} $.
As we will see in the
next section, $\lambda^*$ does indeed exist and admits a simple closed form.

The definitions imply that if an optimal hard threshold exists, then
it gives rise to an optimal truncation value $\hat{r}^*$.
 Formally, it is easy to
verify the following.
\begin{defin} [SVHT Optimal Cutoff Function]
    \label{r_star_lambda_star:lem}
    Let $\sigma>0$. If an optimal hard threshold $\lambda^*$ exists, then 
   	\[
	\hat{r} (\Delta) = \#\{i\in [n] {\ } | {\ } y_i >\lambda^*  \}
    \]
    is an optimal truncation value.	
\end{defin}
~\\
{\em General shrinkage estimators.} 
The SVHT estimators \eqref{eq:SVHTForm} is a special case of a more general
family, which we might call singular value shrinkage estimators. For a
non-decreasing function
$\eta:[0,\infty)\to[0,\infty)$, define
\begin{eqnarray}
\label{eq:EtaForm}
\Mxh^{\eta}= \sum_{i=1}^{n} \eta(\Shy_i)q_i {\ }\V{u_i}{\ } \V{e_i}^{\T}
\end{eqnarray}
where again $e_i \in \R^p$ and
$q_1,\ldots,q_n\in\left\{ \pm 1 \right\}$. Observe that $\hat{X}^\lambda$ is obtained by
taking $\eta$ to be the hard threshold nonlinearity,
$\eta(y)=y\cdot\mathbf{1}_{y>\lambda}$.

How should one choose the shrinker $\eta$? Suppose that there exists a special
``optimal'' shrinker $\eta^*$ for which the asymptotic loss is always minimized,
regardless of the underlying signal. Formally,
\begin{defin} [Optimal Continuous Estimator]
	\label{eta_sta:def}
    Let $\sigma>0$ and let $Con$ denote the family of continuous shrinkers
    $\eta:[0,\infty)\to[0,\infty)$.  
    If there exists a shrinker $\eta^*\in Con$
    for which 
	\[
L(\Mxh^{\eta^*}|\V{x}) \leq L(\Mxh^{\eta}|\V{x})
    \]
    for any $\eta\in Con$, $d\in \N$ and $x\in \R^d$
    then we call $\eta^*$ an optimal shrinker.
\end{defin}
\noindent Here too we abuse notation by writing $\eta^*$ for
$\eta^*_{\beta,\sigma} (\Delta)$. As we show in the next section, the optimal
shrinker $\eta^*$ does exist and admits a simple form.

\section{Results}
\label{S:3}	

For simplicity, we state our results first for the case where the noise level 
$\sigma$ is known. The case of $\sigma$ unknown is deferred to the end of
the section.\\
As seen in lemma \ref{lemNoiselessConfigMDSGetOptimalSimilarityMat},
classical MDS achieves zero loss, or perfect reconstruction
of the original data, in the noiseless case $\sigma=0$. 
Our first main result is the exact asymptotic loss incurred by classical MDS in
the presence of noise.
\begin{thm}
	\label{thmThmMDSExplicitLoss}
	The asymptotic loss of classical MDS with embedding dimension $r$ is given
	by
		\begin{eqnarray}
	\label{eq:ThmMDSExplicitLoss}
	L(\Mxh^{r}|\V{\Sx}) 
	&\overset{a.s.}{=}&
	\Bigg(\sum_{i=1}^{min(t,r)} \bigg(\sqrt{\Sx_i^2+\sigma^2}-\sqrt{\frac{\Sx_i^4-\beta\cdot\sigma^4}{\Sx_i^2}}\bigg)^2 +\frac{2\cdot\beta\cdot \sigma^4}{\Sx_i^2} +
	\beta\cdot \sigma^2\Bigg) \nonumber \\
	&+&\left(\sum_{j=min(t,r)+1}^{d} \Sx_i^2\right)+(r-t)^{+}\cdot \sigma^2\cdot (1+\sqrt{\beta})^2
	\end{eqnarray}

	where $t=\#\{i\in[d] {\ } s.t.{\ } \Sx_i>\sigma \cdot
	\beta^{\frac{1}{4}}\}$. 
\end{thm}
~\\
As discussed in subsection \ref{Shrinkage} above,  for a specific dataset,
classical MDS is equivalent to
singular value hard thresholding (SVHT) at a  hard threshold that depends on the
data. An obvious way to improve the classical MDS algorithm is to consider a carefully
calibrated choice of hard threshold. 
Our next result shows that, in fact, an asymptotically optimal choice of hard
threshold exists -- and even admits a simple closed form.

\begin{thm}
	\label{thmmdsSVHTShrinker}	
	There exists an unique optimal hard threshold 
	$\lambda^*$ 
	(Definition \ref{lambda_star:def}), and its value 
	 is given by
	\begin{eqnarray}
	\label{eq:SVHTOptimalShrinker}
	\lambda=\sigma\cdot
	\sqrt{\left(\sqrt{a}+\frac{1}{\sqrt{a}}\right)\left(\sqrt{a}+\frac{\beta}{\sqrt{a}}\right)}
	\end{eqnarray}
	where $a$ is the unique positive root of 
	\[-3\cdot a^3+a^2\cdot (2\beta +1) +a\cdot (\beta^2 +6\cdot
	\beta)+\beta^2=0\,.\]
	Moreover, 
	$\lambda^*> \sigma\cdot(1+\sqrt{\beta})$.
\end{thm}
~\\
It now follows from Lemma \ref{r_star_lambda_star:lem} that we have obtained an optimal choice 
of embedding dimension for classical MDS:
\begin{crl}
The quantity	
    \begin{eqnarray}
	\hat{r}^{*}=\#\{i\in [n] : \Shy(\Sx_i) >\lambda^* \}
	\end{eqnarray}
	is an optimal truncation value (Definition \ref{r_star:def}).
\end{crl}
~\\
So far we have shown that an
optimal truncation value $\hat{r}^*$ and the optimal hard
threshold $\lambda^*$ exist.  In fact, an optimal shrinker also exists:

\begin{thm}
	\label{thmmdsOptShrinkerExplicit}
There exists a unique optimal shrinker $\eta^*$ (Definition \ref{eta_sta:def}) given by
	\begin{eqnarray}
	\label{eq:mdsOptShrinkerProof}
	\eta^* (\Shy) \overset{a.s.}{=}
	\left\{ \begin{array}{ll}
		\sigma \sqrt{(\Sx(\Shy)/\sigma)^2-\beta-\frac{\beta\cdot
		(1-\beta)}{(\Sx(\Shy)/ \sigma  )^2+\beta}} & 
	\,\, \Shy>\sigma\cdot(1+\sqrt{\beta})\\
	0 & \,\,otherwise
	\end{array} \right.
	\end{eqnarray}
	where 
	\begin{eqnarray}
	\label{eq:ProofMdsOptShirnkerProofNoisedToOrig}
	\Sx(\Shy)=\frac{\sigma}{\sqrt{2}} \sqrt{ \left(\frac{\Shy}{\sigma}\right)^2-1-\beta+\sqrt{\left(\left(\frac{\Shy}{\sigma}\right)^2-1-\beta\right)^2-
	4\beta}}
	\,.
	\end{eqnarray}
\end{thm}

\begin{figure}[H]
	\includegraphics[width=1\linewidth]{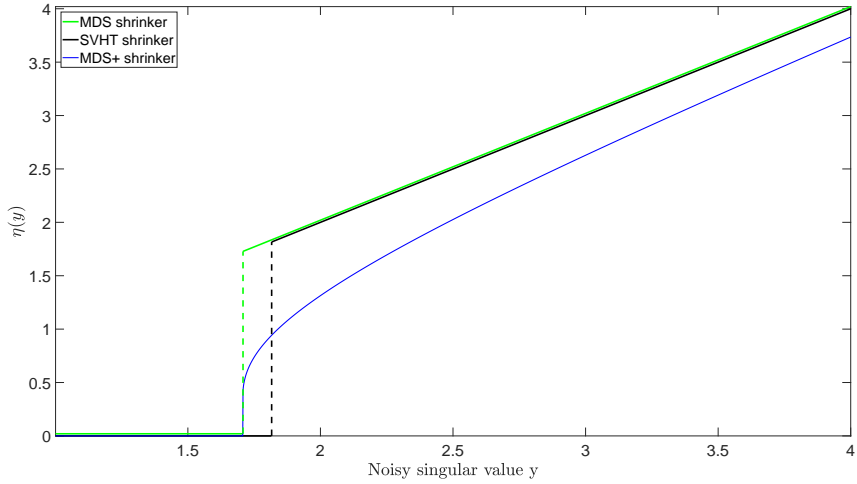}
	\caption{The shrinkers function over the singular value of a rank-1
	    original matrix $\Mx$.
	    Here, $\beta=0.5$ and $\sigma=1$. (Color online)
	\label{fig:shrinkers_funcs}
	}	

\end{figure}
~\\
The estimator $\hat{X}^{\eta^*}$, with the optimal shrinker $\eta^*$ is
a simple alternative to classical MDS, which we call \MDSp. (The algorithm is written explicitly in table \ref{MDS:alg}, located in subsection \ref{SS:1OptVariant})

\noindent We next consider the asymptotic loss obtained by the optimal shrinker $\eta^*$.
\begin{thm}
	\label{thmMDSPLoss}
	The asymptotic loss of the optimal shrinker $\eta^*$ is:
	\begin{eqnarray}
	\label{eq:ProofOptShrinkLoss}
	L(\Mxh^{\eta^*} | \V{\Sx}) \overset{a.s.}{=}
	\beta \cdot \sigma^2 \cdot \left(\sum_{i=1}^t \frac{1-\beta}{(\Sx_i/\sigma)^2+\beta}
	+1 \right)+\sum_{i=t+1}^d \Sx_i^2
	\end{eqnarray}
	where $t=\#\{i\in [d] {\ } s.t. {\ } x_i>\sigma\cdot
	\beta^{\frac{1}{4}}\}$.
\end{thm}

Our next main result quantifies the regret for using 
the classical
MDS (even with optimally tuned with the optimal truncation value $\hat{r}^*$) 
instead of the the proposed algorithm \MDSp.

\begin{thm}
	\label{thmOptBetterTSVD}
	Let $r\in\mathbb{N}$ and consider the classical MDS with embedding
	dimension $r$.
	The asymptotic loss of \MDSp~ is a.s. better then the asymptotic loss of
	classical MDS, and in fact the quantity
	\begin{eqnarray*}
	L(\Mxh^{r} | \Mx)-L(\Mxh^{\eta^*} | \Mx) & \overset{a.s.}{=} &
	\Bigg( \sum_{i=1}^{t} 
	 1_{[i\leq r]} 
	 \bigg[
	\bigg(\sqrt{x_i^2+\sigma^2}-\sqrt{\frac{x_i^4-\beta\cdot \sigma^4}{x_i^2}}\bigg)^2 \\
	&+&\beta\cdot \sigma^2\cdot 
	\frac{(x_i/\sigma)^2\cdot (1+\beta)+2\beta}{(x_i/\sigma)^4+\beta \cdot(x_i/\sigma)^2}
	\bigg]  \\ 
 	 &+&\mathbf{1}_{[i> r]} \cdot \sigma^2\cdot 	\frac{(x_i/\sigma)^4-\beta}
	{(x_i/\sigma)^2+\beta} 
	\Bigg)\\
	 &+&\sum_{i=t+1}^{r} \sigma^2\cdot (1+\sqrt{\beta})^2
	\\
	&\geq& 0 	\,,
	\end{eqnarray*}
	is always non-negative. Here,  $t=\#\{i\in [d] {\ } s.t. {\ } x_i>\sigma\cdot
	\beta^{\frac{1}{4}}\}$.
\end{thm}

Figure \ref{regret} shows the regret over the signal singular value $x$ for
specific values of $r$,$\beta$ and $\sigma$.

\begin{figure}[H]
	\includegraphics[width=1\linewidth]{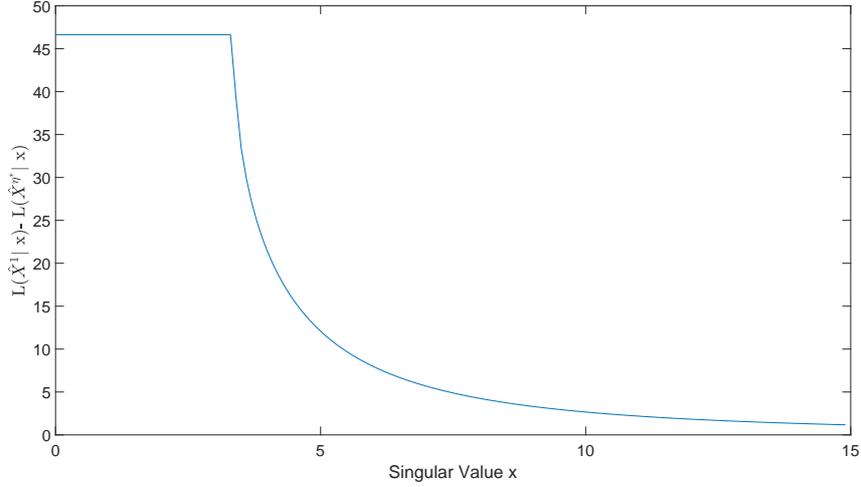}
	\caption{Regret, in terms of asymptotic loss, for using 
	    MDS (with embedding dimension set at $1$) instead of $\MDSp$. 
			Here, $r=1$, $\sigma=4$  and $\beta =0.5 $
			\label{regret}
}

\end{figure}

~\\
{\em Estimating the noise level $\sigma$.}
When the noise level $\sigma$ is unknown, it must be estimated in order to use
\MDSp.
A number of approaches were developed over the years for estimating the noise level
\cite{Shabalin2013,passemier2013variance,kritchman2009non}.
Here we follow the proposal of \cite{Donoho2013_4_3}, which showed:

\begin{thm}
	\label{thmSigmaEstimation}
	Consider
	\begin{eqnarray}
		\label{eq:SigmaEstimation}
		\hat{\sigma}(S)=\sqrt{\frac{s_{med}}{\mu_{\beta}} }
	    \end{eqnarray}
	where $s_{1}\geq \ldots \geq s_{min(n,p)}\geq 0 $ are the  eigenvalues of S, and $s_{med}$ is  their median. Denote the median of the Marcenko Pastur (MP) distribution \cite{Marcenko1967} for $\beta$ by $\mu_{\beta}$.
	Then 
	$ \hat{\sigma}^2(S_n)\overset{a.s.}{\longrightarrow} \sigma^2$ as
	$n\to\infty$.
\end{thm}

The MP median is not available analytically yet is very simple to calculate
numerically (see the Code Supplement \cite{CODE}). It is easy to verify that 
by plugging in $\hat{\sigma}$ for $\sigma$, 
the main results above hold.

Figure \ref{losses:fig} compares the asymptotic loss of classical MDS, optimally
tuned SVHT, and \MDSp.

\begin{figure}[H]
	\includegraphics[width=1\linewidth]{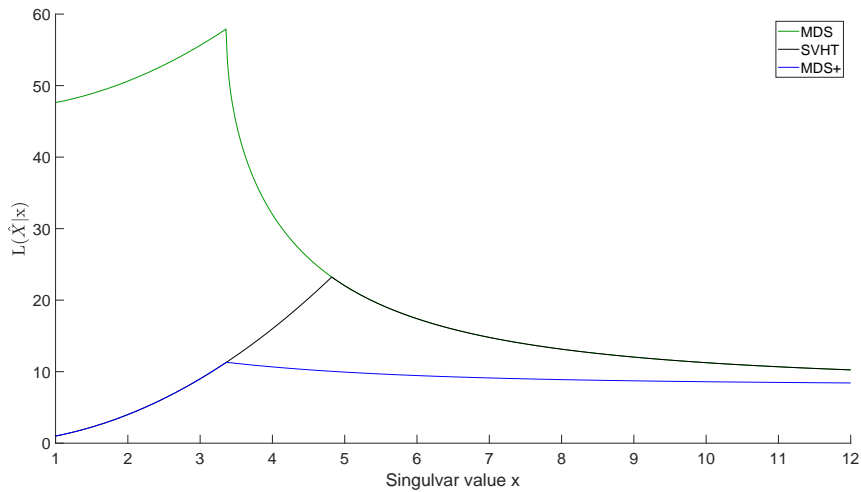}
	\caption{The asymptotic loss of classical MDS (with embedding dimension
	    set to $1$), optimally tuned SVHT and \MDSp.
	    Here $\sigma=4$, $p=500$ and $n=251$. 
	(Color online) \label{losses:fig}}
\end{figure}

{\em The case $\beta>1$.} So far, we only 
    considered the case $\beta \in(0,1]$. The  Appendix contains 
similar results  for the case
$\beta\geq1 $. 

\section{Discussion}
\label{S:5}

\noindent {\em Implications on Manifold Learning.} 
Manifold learning and metric learning methods seek to
	reconstruct a global, low-dimensional geometric structure of a dataset from
	pairwise dissimilarity or affinity measurements \cite{tenenbaum2000Isomap,	roweis2000LLE, coifman2006diffusion, vincent2008extracting, bishop1998developments,
	bellet2013survey,bengio2013representation}. 
	As such, they are specifically designed to be applied to data in
	high-dimensional Euclidean spaces. However, the manifold learning 
	literature contains very little reference to the sensitivity of these
	methods to measurement noise, and particularly to ambient noise
	contaminating the data observed in high-dimensional space.

	The results of the present paper show conclusively that the effect of
	measurement noise cannot be ignored. Indeed, we have shown that the
	behavior of MDS, arguably the earliest and one of the 
	most widely-used Euclidean embedding techniques, and a linear precursor to
	manifold 
	learning techniques, depends crucially on the measurement noise level
	$\sigma$.  For instance, Theorem \ref{thmThmMDSExplicitLoss} shows that
	classical MDS breaks down whenever any one of the singular values $x$ of
	the signal falls below the critical point $\sigma\cdot \beta^{1/4}$.

	These phenomena necessarily hold in any manifold learning technique
	which relies on spectral decomposition of similarity distances, which is
	to say, in basically any manifold learning technique. In this regard our
	results call for a thorough investigation of the noise sensitivity of many well
	known 
	manifold learning methods in the presence of noisy, high dimensional
	measurements. We expect that the phenomena formally quantified in this
	paper, including breakdown of the method in a critical noise level, are
	all present in basically any manifold learning method.
\\~\\
	\noindent{\em Formal quantification of embedding quality.}
	Dimensionality reduction and manifold learning techniques are
	non-supervised. As such, the literature has traditionally ignored their
	formal operating characteristics and focused on asymptotic theorems showing
	that certain manifold quantities are recovered, as well as 
	examples where they appear to
	perform well, relying on visualizations to demonstrate how a method of
	interest can be expected to perform. The present paper takes a
	decision-theoretical approach, and evaluates the performance of a
	non-supervised learning method (in this case, MDS) using a loss
	function. We place our choice of loss function on solid footing with a
	combination of three results:
	\begin{enumerate}
	    \item In the absence of noise, the classical
	MDS algorithm is recovered by minimizing the proposed loss 
	(Lemma \ref{lemNoiselessConfigMDSGetOptimalSimilarityMat}).
    \item The proposed loss function satisfies invariance properties which any
	reasonable loss function should satisfy
   (Lemma  \ref{lemSimlarityDistanceFromAxioms}).
    \item The loss function is based on a pseudo-metric (Lemma
	\ref{pseudo_metric:lem}).
	\end{enumerate}

	By introducing a loss function, analysis and comparison of different
	methods become possible and even simple. It also gives rise to the
	notion of an optimal method.

~\\
\begin{figure}[H]
	\includegraphics[width=1\linewidth]{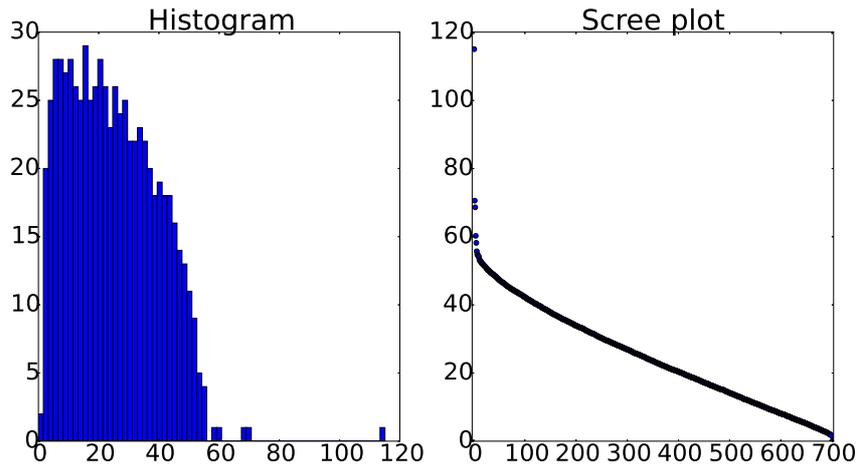}
	\caption{Histogram of singular values (left) and Scree plot of singular
	    values (right) of MNIST data, see Figure \ref{fig:mnist} above.
	    Here the noise level we used is $\sigma=30$, implying noise standard
	    deviation $\sigma/\sqrt{p}=30/\sqrt{784}\approx 1.07$.
		\label{hist_scree}}
\end{figure}

\noindent {\em Bulk edge, Scree plot and the optimal threshold $\lambda^*$.}
It is interesting to compare quantitatively 
the SVHT method with optimal threshold $\lambda^*$  from
Theorem \ref{thmmdsSVHTShrinker} with the classical Scree plot method. While the Scree
plot method itself is not a formally specified algorithm (it is actually more of a
subjective a visual ceremony) we argue that it is roughly equivalent to 
hard thresholding of the singular values (as in \eqref{Xhat_MDS:eq} and \eqref{eq:TSVDForm}), with a specific
choice of threshold. Plotting a histogram of the singular values (Figure
\ref{hist_scree}, left panel) 
instead of their Scree plot (right panel) one observes that 
the so-called ``bulk'' of singular values in lower part of the histogram is the
density of noise-only singular values, known as the
Quarter Circle distribution \cite{bai2010spectral}, with compactly
supported density 
\begin{eqnarray} \label{eq:QC}
f_{QC}(x)= \left\{\begin{array}{cc}
\frac{\sqrt{4\beta\sigma^4 -(x^2-\sigma^2-\beta\sigma^2)^2}}{\pi \sigma^2 \beta \cdot x} & \text{if }x\in [\lambda_-, \lambda_+]\\
0 & otherwise
\end{array}
\right.
\end{eqnarray}
where $\lambda_{\pm}=\sigma\cdot (1+\sqrt{\beta})$ and $\beta\in (0,1]$. The upper
edge of the support, known as the ``bulk edge'' is located at $\lambda_+$ 
Comparing histograms of various noise distributions with their respective Scree
plots, one easily observes that the famous ``knee'' or ``elbow'' in Cattell's Scree
plot \cite{Cattell1966} is simply the location of
the bulk edge. It follows that the Scree plot ceremony is roughly equivalent to
an attempt to visually locate the bulk edge of the underlying noise-only
distribution. 

It is interesting to observe that the optimal threshold $\lambda^*$ is always
larger than the bulk edge. In other words, even when the singular values are
visually recognizable above the elbow in the Scree plot, and are detectable as
signal  singular values (distinct from the noise singular values) 
it is still worthwhile,
from the perspective of the asymptotic loss, to exclude them from the
reconstruction, as long as they are ``too close'' to the bulk edge. For a  more
thorough discussion of this phenomena,  see \cite{Donoho2013_4_3}.

Figure  \ref{simulation} compares the asymptotic loss of the classical MDS (with
the Scree plot ceremony, namely bulk-edge hard thresholding), MDS with the
optimal threshold $\lambda^*$ and MDS+ -- equations \eqref{Xhat_MDS:eq}, \eqref{eq:SVHTOptimalShrinker} and
\eqref{eq:mdsOptShrinkerProof} respectively,
over the signal singular value  $x$. 

\begin{figure}[h]
\includegraphics[width=1\linewidth]{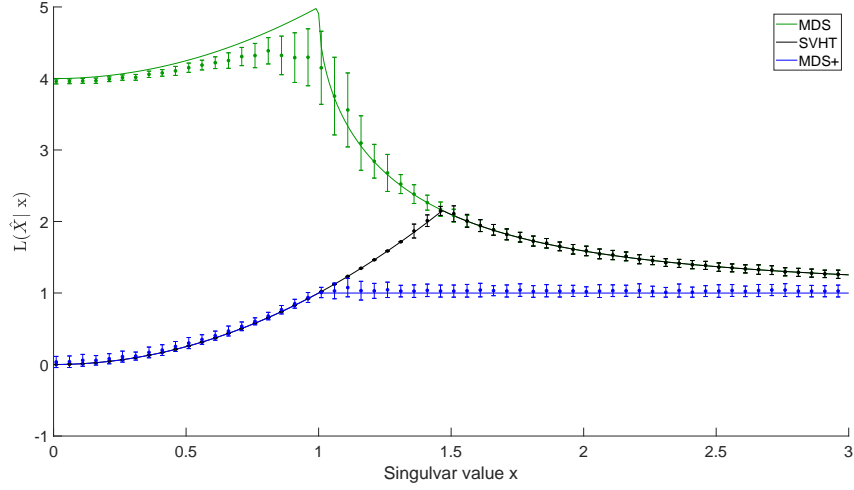}
	\caption{
		$X\in M_{1000\times1000}$ , $d=1$, $r=1$, $\sigma=1$. The plot is composed of 100 tests for each singular value. The points indicates the mean of the loss over the tests, while the error bars indicate the standard deviation of the tests. The solid line indicates the theoretic loss  (Color online)}
	\label{simulation}
	
\end{figure}

~\\
{\em Accuracy in finite-$n$.}
The threshold and shrinker derived in this paper are based on an asymptotic
loss, and one naturally wonders whether they remain approximately optimal in
finite values of $n$. While this question is beyond our present scope, we note
that the convergence to the limit of the loss and its underlying components is
rapid and indeed the asymptotically optimal threshold and shrinker can be used
in practice. Figure \ref{simulation} shows the predicted asymptotic loss (solid lines)
and the empirically observed mean and standard deviation from a Monte-Carlo
simulation. Precise evaluation of the finite-$n$ effects will be evaluated
elsewhere.

\section{Proofs} \label{S:6} 
\subsection{Notation}
Following subsection \ref{AsymptoticModel}, we assume the same $d$ singular values over all $\Mx_n$. Denote its singular values by $\V{x}=(x_1,\ldots,x_d)$, as assumed they 
are non-degenerate, so that $\Sx_1>\ldots> \Sx_d> 0$, 
and they are mean centered,
so that $H\cdot \Mx_n=\Mx_n$.
Denote its left and right singular vectors as
$\{\V{v_{n,1}},\ldots,\V{v_{n,n}}\}\subset \R^n$ and
$\{\V{\tilde{v}_{n,1}},\ldots,\V{\tilde{v}_{n,d}}\} \subset \R^d$, 
respectively. Denote the corresponding matrices of singular vectors by
$V_n=[\V{v_{n,1}},\ldots,\V{v_{n,n}}]\in O(n)$ 
and $\tilde{V}_n=[\V{\tilde{v}_{n,1}},\ldots,\V{\tilde{v}_{n,d}}]\in O(d)$. 
Similarly,
denote the singular values of $H\cdot Y_n$ by 
$\Shy_{n,1}\geq \ldots \geq \Shy_{n,n}\geq 0$ and its left and right singular vectors by $\{ \V{u_{n,1}},\ldots,\V{u_{n,n}} \}\in \R^d$ and $\{ \V{\tilde{u}_{n,1}},\ldots,\V{\tilde{u}_{n,n}} \}\in \R^{p_n}$, respectively.
Let the corresponding singular vector matrices be $U_n=[\V{u_{n,1}},\ldots,\V{v_{n,n}}]\in O(n)$ and $\tilde{U}_n=[\V{\tilde{u}_{n,1}},\ldots,\V{\tilde{u}_{n,p_n}}]\in O(p_n)$.

Some properties of the centering matrix H are important in what follows.
Denote $H= U_H \cdot D_H  \cdot U_H ^\T$ as its eigenvalue Decomposition, 
where $U_H\in O(n)$ is the eigenvector matrix and $D_H=diag(1,\ldots,1,0)$ 
is the eigenvalue matrix. 
Observe that $(U_H)_{*,n}=\frac{1}{\sqrt{n}}\V{1_n}$ and that 
$H\cdot H=H$ as the centering matrix is a projection. While $H$ depends on $n$, we suppress it and leave it to be inferred from context.

\subsection{Limiting location of the singular values and vectors}

\begin{lem}
	\label{lem:MPPropHY} 
	The asymptotic distribution of $H\cdot \My$ has the following properties for any finite $i$ and $j$:
	\begin{enumerate}
		\item $\underset{n\rightarrow\infty}{lim}\Shy_{n,i} \overset {a.s.} {=} \left\{ \begin{array}{ll}
		\Shy(\Sx_i) & i\in[t]\\
		\sigma\cdot (1+\sqrt{\beta}) & otherwise
		\end{array} \right.$ \\
		\item $\underset{n\rightarrow\infty}{lim}| \langle\V{u_{n,i}}, \V{v_{n,j}}\rangle |\overset{a.s.}{=} \left\{ \begin{array} {ll}
		c(x) & i=j\in [t] \\
		0 & otherwise
		\end{array}
		\right. $
	\end{enumerate}
	when $t \equiv\#\{i\in[d] {\ }| \Sx_i/\sigma>\beta^{\frac{1}{4}}\}$ and 
	\begin{eqnarray*}
		c(x) &=& \sqrt{\frac{(x/\sigma)^4-\beta}{(x/\sigma)^4+\beta (x/\sigma)^2}} \\
	y(x) &=& \sqrt{(x/\sigma+\frac{1}{x/\sigma})\cdot(x/\sigma+\frac{\beta}{x/\sigma})}
	\end{eqnarray*}
\end{lem}

\noindent In order to prove Lemma  \ref{lem:MPPropHY} we'll need the following result:
\begin{lem}
	\label{lem:HY_Using_X}
	We have
	\begin{eqnarray}
	\label{eq:ambientDimDistAs}
	H\cdot \My =  U_H\cdot  \begin{bmatrix}
	\sum_{i=1}^d \Sx_i \cdot (I_{n-1\times n}\cdot U_H^\T\cdot \V{v_i}) \cdot [\V{\tilde{v}_i}^\T, \V{0_{p-d}^\T}] + \frac{1}{p} \mathring{Z}_{1:n-1,*} \\
	\V{0_{d}^\T}
	\end{bmatrix}\cdot R
	\end{eqnarray}
	Where $\mathring{Z}\sim Z$, and $U_H\in O(n)$ is the eigenvectors matrix form of the Eigenvalue Decomposition of H.
\end{lem}
\begin{proof}
	Another way of writing $H\cdot \My$ is:
	\begin{eqnarray}
	H\cdot \My= H\cdot \big( Y\cdot I_{d \times p} \cdot R+ \frac{1}{p}Z \big)= U_H\cdot D_H \cdot U_H^\T \cdot \big( H\cdot Y\cdot I_{d \times p}+ \frac{1}{p} \grave{Z} \big)\cdot R
	\end{eqnarray}
	~\\
	where $\grave{Z}\equiv Z\cdot R^\T$. \eqref{eq:ambientDimDistAs} follows since
	\begin{eqnarray*}
		&D_H \cdot U_H^\T \cdot \big( H\cdot Y\cdot I_{d \times p}+ \frac{1}{p} \grave{Z} \big)=
		D_H \cdot \big( \sum_{i=1}^d \Sx_i \cdot (U_H^\T\cdot \V{v_i}) \cdot [\V{\tilde{v}}_i^\T, \V{0_{p-d}^\T}] + \frac{1}{p} \mathring{Z} \big)\\
		&=  \begin{bmatrix}
			\mysum{i}{1}{d} \Sx_i \cdot (I_{n-1\times n}\cdot U_H^\T\cdot \V{v_i}) \cdot [\V{\tilde{v}_i}^\T, \V{0_{p-d}^\T}] + \frac{1}{p} \mathring{Z}_{1:n-1,*} \\
			\V{0_{p}^\T}
		\end{bmatrix} 
	\end{eqnarray*}
	where $\mathring{Z} \equiv U_H\cdot \grave{Z}$. $\mathring{Z} \sim Z$  follows from the  invariant white noise property of the framework, as defined in subsection \ref{AsymptoticModel}.
\end{proof}

\begin{proof}[Proof of Lemma \ref{lem:MPPropHY}]
	Following lemma \ref{lem:HY_Using_X}
	\begin{eqnarray*}
	U_H^\T \cdot H \cdot \My_n \cdot R_n^\T = \begin{bmatrix}
	\mysum{i}{1}{d} \Sx_i \cdot (I_{n-1\times n}\cdot U_H^\T\cdot \V{v_{n,i}}) \cdot [\V{\tilde{v}_{n,i}}^\T, \V{0_{p_n-d}^\T}] + \frac{1}{p_n} (\mathring{Z}_n)_{1:n-1,*} \\
	\V{0_{p_n}^\T} 
	\end{bmatrix} 
	\end{eqnarray*}
	
	First we define $W_n\equiv \sum_{i=1}^d \Sx_i \cdot (I_{n-1\times n}\cdot U_H^\T\cdot \V{v_{n,i}}) \cdot [\V{\tilde{v}_{n,i}}^\T, \V{0_{p_n-d}^\T}] \in M_{n-1 \times p_n}$ and show that this is an acceptable SVD decomposition of  $W_n$.\\
	
	We start by showing that $\{I_{n-1\times n}\cdot U_H^\T\cdot \V{v_{n,i}}\}_{i=1}^d$ is an orthonormal set, where $d=rank(\Mx_n)$. 
	Now, we can see that 
	$(U_H)_{*,n}=\frac{1}{\sqrt{n}}\V{1_n} \in$ span $\{\{\V{v_{n,k}}\}_{k=d+1}^n\}$ following the assumption that $\Mx_n=H\cdot \Mx_n$, and $rank(\Mx_n)=d$. Therefore  for any $i\in \{1,\ldots,d\}$
	
	\begin{eqnarray}
	\label{eq:lastSingVec}
	\langle\V{e_n},  U_H^\T \cdot \V{v_{n,i}}\rangle= \langle\frac{1}{\sqrt{n}}\V{1_n},\V{v_{n,i}}\rangle =0
	\end{eqnarray}
	since $\{\V{v_{n,i}}\}$ is an orthogonal set. Therefore for any $i,j\in \{1,\ldots,d\}$ 
	:
	\begin{eqnarray*}
		<I_{n-1\times n}\cdot U_H^\T \cdot \V{v_{n,i}} , I_{n-1\times n}\cdot U_H^\T\cdot \V{v_{n,j}}>&=& \\
		< U_H^\T \cdot \V{v_{i,n}} , (I_n-\V{e_n}\cdot \V{e_n}^\T)\cdot U_H^\T\cdot \V{v_{n,j}}>&=&\\
		<U_H^\T \cdot \V{v_{n,i}} , U_H^\T\cdot \V{v_{n,j}}> - 0=<\V{v_{n,i}},\V{v_{n,j}}>&=&1_{[i=j]}
	\end{eqnarray*}
	\\
	The set $\left\{ \left( \begin{array}{c}
	\V{\tilde{v}_{n,i}} \\
	\V{0_{p_n-d}}
	\end{array} \right) \right\}_{i=1}^d$ is an orthonormal set also simply because $\{\V{\tilde{v}_{n,i}}\}_{i=1}^d$ is an orthonormal set.\\
	\\
	Second, we define $Q_n\equiv W_n +\frac{1}{p_n}(\mathring{Z_n})_{1:n-1,*}$ and denote its SVD decomposition form by 
	$$
	Q_n= \sum_{i=1}^{n-1} q_{n,i} \cdot \V{w_{n,i}} \V{\tilde{w}_{n,i}}^\T
	$$
	where $\V{w_{n,i}}\in \R^{n-1}$ and $\V{\tilde{w}_{n,i}}\in \R^{p_n}$ are its i-th left and right singular vectors, respectively, and $q_{n,i}$ is its i-th singular value. We would like to emphasize the relations between $Q_n$ and $W_n$ at the limit.\\
	Denote the singular values of $(\mathring{Z}_n/p_n)_{1:n-1,*}$ by $z_1\geq \ldots \geq z_{n-1}$. 
	Following \cite{bai2010spectral} the density of the singular values  of $\mathring{Z}_n/p_n$  in the limit is the quarter-circle density:
	\[
	f(x)=\frac{\sqrt{4\beta\sigma^4-(x^2-\sigma^2-\beta\sigma^2)^2}}{\pi \beta\sigma^2\cdot x},
	\] 
	following \cite{yin1988limit}  $\Shy_{n,1}\overset{a.s.}{\rightarrow}\sigma\cdot(1+\sqrt{\beta})$, and following \cite{bai2008limit} $\Shy_{n,min(n,p_n)}\overset{a.s.}{\rightarrow}\sigma \cdot(1-\sqrt{\beta})$. These satisfies assumptions 2.1,2.2 and 2.3 in \cite{Benaych-Georges2012}, respectively, therefore for any
	finite $i$:
	\begin{enumerate}
		\item Translation of singular values:
		\begin{eqnarray*}
		\underset{n\rightarrow\infty}{lim} q_{n,i} \overset{a.s.}{=}     
		\left\{\begin{array}{cc}
		y(x_i) & if {\ } x_i/\sigma>\beta^{1/4}\\
		\sigma\cdot (1+\sqrt{\beta}) &  otherwise
		\end{array} \right.
		\end{eqnarray*}
		\item Rotation of the left singular vectors:
		\begin{eqnarray*}
		\underset{n\rightarrow\infty}{lim} |\langle U_H^\T\cdot \V{v_{n,i}} ,\V{w_{n,j}}\rangle| 	\overset{a.s.}{=}
		\left\{ \begin{array}{ll}
		c(\Sx_i) & if{\ } \Sx_i/\sigma>\beta^{\frac{1}{4}} {\ }and{\ } i=j\\
		0 & otherwise
		\end{array} \right.
		\end{eqnarray*}
	\end{enumerate}
	where 
	\begin{eqnarray*}
	c(x) &=& \sqrt{\frac{(x/\sigma)^4-\beta}{(x/\sigma)^4+\beta (x/\sigma)^2}} \\
	y(x) &=& \sqrt{(x/\sigma+\frac{1}{x/\sigma})\cdot(x/\sigma+\frac{\beta}{x/\sigma})}
	\end{eqnarray*}
	Third we denote $W^L_{n} \equiv \begin{bmatrix}  W_n \\\V{0_{p_n}^\T}  \end{bmatrix}$ and $Q^L_{n}\equiv \begin{bmatrix}	Q_n \\ \V{0_{p_n}^\T} \end{bmatrix}$. Following equation \eqref{eq:lastSingVec}, those matrices could be expressed in a singular value decomposition form using $W_n$ and $Q_n$, respectively, by
	\begin{eqnarray*}
		W^L_n &=& \sum_{i=1}^d \Sx_i \cdot \left[\begin{array}{cc}
			(I_{n-1\times n}\cdot U_H^\T\cdot \V{v_{n,i}}) \\ 0
		\end{array}\right] \cdot [\V{\tilde{v}_{n,i}}^\T, \V{0_{p_n-d}^\T}] \\
	&= & \sum_{i=1}^d \Sx_i \cdot 
		(U_H^\T\cdot \V{v_{n,i}}) \cdot [\V{\tilde{v}_{n,i}}^\T, \V{0_{p_n-d}^\T}]\\ \text{ }\\
		Q^L_n &=& \underset{i=1}{\overset{n-1}{\sum}} q_{n,i} \cdot \left[\begin{array}{cc}
			\V{w_{n,i}} \\ 0 
		\end{array}\right] \V{\tilde{w}_{n,i}}^\T 
	\end{eqnarray*} 
	Another valid way of writing $Q^L_n$ in a singular value decomposition components form is
	\begin{eqnarray*}
	Q^L_n=U_H^\T \cdot H\cdot \My \cdot R_n^\T= \sum_{i=1}^n \Shy_i \cdot (U_H^\T\cdot \V{u_{n,i}})\cdot (\V{\tilde{u}_{n,i}}^\T \cdot R_n^\T)
	\end{eqnarray*}
	since  $R_n,U_H$ are orthogonal matrices. Therefore for any finite i holds
	\begin{enumerate}
		\item Translation of singular values:
		\begin{eqnarray*}
		\underset{n\rightarrow\infty}{lim} \Shy_{n,i} \overset{a.s.}{=}
		\left\{\begin{array}{cc}
			y(x_i) & if {\ } x_i/\sigma>\beta^{1/4}\\
			\sigma\cdot (1+\sqrt{\beta}) &  otherwise
		\end{array} \right.
		\end{eqnarray*}
		following the fact that for any n and any i $q_{n,i}=y_{n,i}$.
		\item Rotation of the left singular vectors:
		\begin{eqnarray*}
		\underset{n\rightarrow\infty}{lim}|\langle \V{v_{n,i}}, \V{u_{n,j}} \rangle|&=& 
		\underset{n\rightarrow\infty}{lim} \bigg | \bigg \langle U_H^\T \cdot \V{v_{n,i}},  \left[\begin{array}{cc}
		\V{w_{n,j}} \\ 0
		\end{array} \right] \bigg\rangle \bigg|  \\
		&\overset{a.s.}{=}&
		\left\{\begin{array}{ll}
		c(\Sx_i) & if {\ } \Sx_i>\sigma\cdot \beta^{\frac{1}{4}}{\ }and
		{\ }i=j\\ 0 & otherwise
		\end{array}\right.
		\end{eqnarray*}
		since $\langle \V{v_{n,i}}, \V{u_{n,j}}\rangle = \langle U_H^\T \cdot \V{v_{n,i}}, U_H^\T \cdot \V{u_{n,j}}\rangle$ for any n. \vspace{1mm} 
	\end{enumerate}
\end{proof}

\subsection{MDS loss function analysis}
\noindent 
In the following subsection  the index $n$ is suppressed to simplify notation.

\begin{thm}
	\label{thmMDSRealLoss}
	Let $\Delta\in\mathbb{R}^{n\times n}$ be an Euclidean distance matrix
	over the dataset $\{\Vy_i\}_{i=1}^n$, namely, 
	$\Delta_{i,j}=\parallel \Vy_i - \Vy_j \parallel_2 ^2 $. 
	The similarity matrix \eqref{eq:Sdiag}, used in the classical MDS
	algorithm, satisfies
	\begin{eqnarray}
	\label{eq:mdsSimMatProof}
	S= H\cdot \My \cdot (H\cdot \My)^\T\,.
	\end{eqnarray}
	Moreover, $S_{i,j}= \langle\V{\Vy_i} -\V{\mu_{\Vy}}\,,\, \V{\Vy_j} -
	\V{\mu_{\Vy}}\rangle$, where $\V{\mu_{\Vy}}=\frac{1}{n}\sum_{i=1}^n \V{\Vy_i}$
	is the empirical mean.
	
\end{thm}

\begin{proof}
	Another way of writing the distance matrix is:
\[
	\Delta=diag(\My\cdot \My^\T)\cdot \V{1_n}^\T + \V{1_n}\cdot diag(\My \cdot \My^\T )-2\cdot \My\cdot \My^\T. 
\]

Indeed, this follows from $\Delta_{i,j} = \parallel \V{\Vy_i} \parallel_2^2 +
\parallel \V{\Vy_j} \parallel_2^2 -2 \langle\V{\Vy_i},\V{\Vy_j}\rangle$. Equation 
\eqref{eq:mdsSimMatProof} now follows since

\begin{eqnarray*}
	S &=&-\frac{1}{2}(I_n-\frac{1}{n}\cdot \V{1}\cdot \V{1}^\T)\cdot
	\Delta \cdot (I_n-\frac{1}{n}\cdot \V{1}\cdot \V{1}^\T)\\&=&
-\frac{1}{2}\cdot (I_n-\frac{1}{n}\cdot \V{1}\cdot \V{1}^\T)\cdot(-2\cdot
\My\cdot \My^\T) \cdot (I_n-\frac{1}{n}\cdot \V{1}\cdot \V{1}^\T)\\&=&
H\cdot \My \cdot (H\cdot \My)^\T\,,
\end{eqnarray*}
where we have used	
\[diag(\My \cdot \My^\T)\cdot\V{1}^\T\cdot(I_n-\frac{1}{n}\V{1}\cdot\V{1}^\T)=diag(\My\cdot \My^\T)\cdot\V{1}^\T-diag(\My\cdot \My^\T)\cdot\V{1}^\T=0_{n\times n}.\] 
Now, each entry of the similarity matrix is simply
\[
	S_{i,j}=(H\cdot \My)_{i,*}\cdot (H\cdot \My)_{j,*}^\T=\langle\V{\Vy_i}-\V{\mu_{\Vy}}, \V{\Vy_j} -\V{\mu_{\Vy}}\rangle,
\]
since
\[
(H\cdot Y)_{i,j}=(I-\frac{1}{n}\V{1_n}\V{1_n}^\T)_{i,*} \cdot Y_{*,j} =\\
Y_{i,j}-\frac{1}{n}\sum_{k=1}^n Y_{i,k}= (\V{b_i})_j -\frac{1}{n}\sum_{k=1}^n
(\V{b_k})_j\,.
\] 
\end{proof}

\begin{thm}
	\label{thmNoiselessMDSPerfectRecon}
	Let $\sigma=0$ and the embedding dimension be $r=d$. The MDS embedding result has the following property:
	\begin{eqnarray}
	\label{eq:thmNoiselessMDSPerfectRecon}
	\V{\Vxh_i}= S\cdot \V{\Vx_i}
	\end{eqnarray}	
	for some $S\in O(d)$, where $\V{\Vxh_i}$ is the reconstruction of $\V{\Vx_i}$.\\
\end{thm}
\begin{proof}
	The noiseless configuration is
	\[
	\My= [\Mx, 0_{d\times p}] \cdot R
	\]
	It follows that $H\cdot Y =Y$ since we have assumed that $H\cdot X= X$.
	Following Theorem \ref{thmMDSRealLoss}, the similarity matrix could be described using $\Mx$'s SVD decomposition components:
	\[
	S= (HY)\cdot (HY)^{\T}= (H\cdot [\Mx, 0_{d\times p}] \cdot R)\cdot (H\cdot [\Mx, 0_{d\times p}] \cdot R)^{\T}= X\cdot X^\T= \sum_{i=1}^{d} x_i \cdot \V{v_i} \V{v_i}^\T
	\]
 
	It follows from \eqref{Xhat_MDS:eq} that
	the MDS algorithm (with embedding dimension $d$) results in 
	\[
	\hat{X}^{MDS}= \sum_{i=1}^d x_i q_i \cdot \V{v_i} \V{e_i}^\T\in M_{n\times d}
	\]
	where $q_1,\ldots,q_d\in\{\pm 1 \}$. 
	Equation
	\eqref{eq:thmNoiselessMDSPerfectRecon} now follows while plugging in
	\\$S^\T=\tilde{V}\cdot diag(q_1,\ldots,q_d)\in O(d)$ as the orthogonal matrix
\begin{eqnarray*}
	\Mxh^{MDS}= \Mx \cdot S^\T
\end{eqnarray*}
or in its vector form
\begin{eqnarray*}
 \hat{a}_i= S \cdot a_i \qquad \text{  i=1,\ldots,r}
\end{eqnarray*}	
\end{proof}

\begin{proof}[Proof of Theorem \ref{lemNoiselessConfigMDSGetOptimalSimilarityMat}]
	As seen in the proof of Theorem \ref{thmNoiselessMDSPerfectRecon}, $H\cdot \My$ could be written using its SVD decomposition components
	\begin{eqnarray*}	
	\label{eq:lemNoiselessConfigMDSGetOptimalSimilarityMat}
	H\cdot Y=Y=\sum_{i=1}^{d} x_i\cdot \V{v_i} \cdot \left( R^\T \cdot \left( \begin{array}{c}
	\V{\tilde{v}_i} \\ 
	\V{0_{p-d}}
	\end{array} 	\right) \right)^\T
	\end{eqnarray*}
	
	By Theorem \ref{thmMDSRealLoss}, in this case we have
	$rank(S)=rank(H\cdot Y)= d$.
	Denote the MDS embedding for any $r$ by $\Mxh^{MDS}_r$. 
	It follows from \eqref{Xhat_MDS:eq} that
	the the result of the MDS embedding is
	\[
	\Mxh^{MDS}_r= \sum_{i=1}^{min(r,d)}\Sx_i  q_i \cdot \V{v_i} \V{e_i}^\T
	\]
	where $e_i\in \R^r$ and $q_i\in \{\pm 1\}$ for any $i\in \{1,\ldots,r\}$. 	Meaning that for any $r_1,r_2\geq d$ holds
	\begin{eqnarray*}
	M_n(\Mxh^{MDS}_{r_1},X)= M_n(\Mxh^{MDS}_{r_2},X)=0
	\end{eqnarray*}
	Therefore the MDS's embedding result is in the set of optimal embeddings that  minimizes $M_n(\cdot ,X)$ since $M_n(\cdot ,\cdot )\geq 0$.
\end{proof}

\begin{proof}[Proof of Theorem \ref{lemSimlarityDistanceFromAxioms}]
	The rotation invariance property is a direct result of the minimizers' domain, which is a group that is closed under multiplication. The translation invariance property is a direct result of the similarity distance definition where the mean of each dataset is taken off the dataset before trying to find the optimal minimizer. 
	This function is invariant for zero- padding as a direct result of the zero padding property of the actual definition, meaning that $M_n(Y_1,Y_2)= M_n([Y_1, 0_{n\times k}, Y_2])$ for any k. Now, the padding invariance is achieved using the zero-padding invariance property and the translation invariance discussed above. Meaning, for any k and any $c\in \R^k$ holds
	\begin{eqnarray*}
	M_n\left( \left[Y_1, \begin{array}{c}
	c^T \\ c^T \\ \ldots \\ c^T
	\end{array}\right] , Y_2 \right) =
	M_n\left( \left[Y_1,0_{n\times k}\right] , Y_2 \right)  = M_n(Y_1, Y_2)
	\end{eqnarray*}
\end{proof}

\begin{proof}[Proof of lemma \ref{pseudo_metric:lem}]
	The first property holds, by simply using $R=I_d$. The second property can be derived as follows 
	\begin{eqnarray*}
	M_n(X,Y)&=&
	\underset{R\in O(n)}{min} \| X- Y\cdot R \| \\
	&=&\underset{R\in O(n)}{min} \| X\cdot R^T- Y \| \\ 
	&=&\underset{Q\in O(n)}{min} \| X\cdot Q- Y \| \\
	&=&M_n(Y,X)
	\end{eqnarray*}
	since $O(n)$ is a group that is closed under multiplication and $R^{-1}=R^\T$ for any $R\in O(n)$.
	For the triangle inequality property we are going to derive a similar property for start
	\begin{eqnarray*}
	M_n(X,Y)&=& \underset{R\in O(n)}{min} \| X- Y\cdot R \|\\ 
	&=&\underset{R\in O(n)}{min} \| X-Z\cdot Q +Z\cdot Q- Y\cdot R \|  \\
	&\leq& \underset{R\in O(n)}{min} \| X-Z\cdot Q\| +\|Z\cdot Q- Y\cdot R \| \\
	&=& \| X-Z\cdot Q\| +\underset{R\in O(n)}{min}\|Z- Y\cdot R\cdot Q^\T \| \\
	&=& \| X-Z\cdot Q\| +\underset{W\in O(n)}{min}\|Z- Y\cdot W \|
	\end{eqnarray*}
	where $Q\in O(n)$.	The triangle inequality property follows since 
	\begin{eqnarray*}
	M_n(X,Y)\leq \| X-Z\cdot Q\| +\underset{W\in O(n)}{min}\|Z- Y\cdot W \| = M_n(X,Z) + M_n (Z,Y)
	\end{eqnarray*}
	where 
	$Q$ is the minimizer of $M_n(X,Z)$.
\end{proof}

\subsection{MDS loss function estimation}

\begin{lem} 
	\label{lem:MDSLossLemma}
	 Let $d\in \N$. For any matrix $\Mx\in M_{n\times d}$, without any degenerate positive singular values, and any embedding dimension $r\in\N$, the loss of the MDS estimator is:
	
	\begin{eqnarray}
	\label{eq:lemMDSLoss}
	L(\hat{X}^{r}|\V{\Sx})
	\overset{a.s.}{=}
	\sum_{i=1}^{r} L_{TSVD} \big(\Sx_i , \Shy(\Sx_i)\big)
	\end{eqnarray}
	where
$L_{TSVD}(\Shy(\Sx_i), \Sx_i) = \left\{ \begin{array}{cc}
	\Shy(\Sx_i)^2+ \Sx_i^2 -2\cdot \Sx_i \cdot \Shy(\Sx_i) \cdot c(\Sx_i) & i\leq min(t,r) \\
	\sigma^2\cdot (1+\sqrt{\beta})^2\cdot 1_{[i\leq r]} + \Sx_i^2\cdot
	1_{[i\leq d]}  &  otherwise
\end{array} \right. \\
,$ and
$t = \#|\{i\in[d] {\ }$ s.t. $\Sx_i/\sigma>\beta^{\frac{1}{4}}\}|$.
\end{lem}

\begin{proof}
As seen in section \ref{Shrinkage}, the TSVD and the MDS embedding would result in the same similarity distance from any matrix. The loss of the TSVD algorithm embedding into r dimensions over some specific n is
\begin{eqnarray*}
M_n^2(\Mxh^{r},\Mx_n) &=&
\underset{R\in O( p )}{min}
\parallel \mysum{i}{1}{r} \Shy_{n,i}q_{n,i} \cdot \V{u_{n,i}} \cdot \V{e_i}^\T\cdot R- 
\mysum{j}{1}{d} \Sx_j \cdot \V{v_{n,j}} \cdot \left[\begin{array}{c}
\V{\tilde{v}_{n,j}} \\ \V{0_{p_n-d}}
\end{array}\right]^\T\parallel_F^2 \\
&=&\underset{R\in O(p)}{min} \left(\mysum{i}{1}{d} \Sx_{n,i}^2\right) +\left(\mysum{j}{1}{r} \Shy_{n,i}^2\right) \\
&-&2\cdot \mysum{i}{1}{r} \mysum{j}{1}{d} \Sx_j \cdot \Shy_i\cdot q_{n,i}\cdot \langle\V{u_{n,i}}, \V{v_{n,j}} \rangle \bigg\langle R^{\T}\cdot \V{e_i} ,\left[\begin{array}{c}
\V{\tilde{v}_{n,j}} \\ \V{0_{p-d}}
\end{array}\right] \bigg\rangle
\end{eqnarray*}
where $e_i\in \R^{p_n}$ and $q_{n,i}\in \{\pm 1\}$ for any i and n. The loss of the TSVD algorithm with an embedding dimension of r considers an asymptotic configuration where $n,p\rightarrow \infty$. Following lemma \ref{lem:MPPropHY} 
\begin{eqnarray*}
L(\Mxh^{TSVD}| \V{\Sx})
&=& \underset{n\rightarrow \infty}{lim}
M_n^2(\Mxh^{TSVD},\Mx_n) \\
&\overset{a.s.}{=}&\left(\sum_{i=1}^{d} \Sx_i^2\right) +\left(\sum_{j=1}^{r} \Shy(\Sx_i)^2\right) -2\cdot \sum_{i=1}^{min(r,t)} \Sx_i \cdot \Shy(\Sx_i)\cdot  c(x_i)
\\
&=&\left(\sum_{i=1}^{min(r,t)} \Sx_i^2+ \Shy(\Sx_i)^2 -2\cdot \Sx_i \cdot \Shy( \Sx_i)\cdot c(\Sx_i)\right) \\
&+&\left(\sum_{i=min(r,t)+1}^{r}\sigma^2\cdot (1+\sqrt{\beta})^2 \right)+    \left(\sum_{i=min(r,t)+1}^{d} \Sx_i^2\right)
\end{eqnarray*}
where the minimizer in the limit of $n\rightarrow \infty$ is:
\begin{eqnarray}
\label{eq:optimalR}
 R_n= \tilde{V}_n^\T\cdot diag\left( w_{n,1},\ldots, w_{n,max(r,d)},1,\ldots,1\right)
\end{eqnarray}
and $w_{n,i}\equiv q_{n,i}\cdot sign(\langle u_{n,i},v_{n,i}\rangle )$ for any n and  $i\in\{1,\ldots,max(r,d)\}$. This follows from the fact that $x_i,y(x_i)>0$ for any n and $i\in \{1,\ldots,min(r,t)\}$,  and the fact that $|\langle R^\T \cdot e_i, \tilde{v}_{n,i}\rangle |\leq 1$ for any n and i and for any $R\in O(p_n)$.\\

\end{proof}

\begin{lem}
	\label{lem:MDSLossLemmaOverBulk}
	Given the same formulation as in the previous lemma \ref{lem:MDSLossLemma}.
	The explicit function of $L_{TSVD}$ for any $i\leq min(t,r)$ is:
	\begin{eqnarray}
	\label{eq:MDSLossLemmaOverBulk}
	L_{TSVD} \big(\Shy(\Sx_i), \Sx_i \big) = \left(\sqrt{\Sx_i^2+\sigma^2}-\sqrt{\frac{\Sx_i^4-\beta\cdot\sigma^4}{\Sx_i^2}}\right)^2 +\frac{2\cdot\beta\cdot \sigma^4}{\Sx_i^2} +
	\beta\cdot \sigma^2 
	\end{eqnarray}
\end{lem}

\begin{proof}
	Following lemma \ref{lem:MPPropHY} and lemma \ref{lem:MDSLossLemma},  we can derive \eqref{eq:MDSLossLemmaOverBulk}:
{
\allowdisplaybreaks
	\begin{eqnarray*}	
L_{TSVD}\big(\Sx,\Shy(\Sx)\big)  &=&
	\Sx^2 +\Shy(\Sx)^2-2\cdot \Sx \cdot \Shy(\Sx) \cdot c(\Sx) 
	\\
	&=&\Sx^2+\sigma^2\cdot\big((\Sx/\sigma+\frac{1}{\Sx/\sigma})\cdot(\Sx/\sigma+\frac{\beta}{\Sx/\sigma})\big) \notag\\
	&-&2\cdot \Sx\cdot \sigma\cdot \sqrt{(\Sx/\sigma +\frac{1}{\Sx/\sigma})(\Sx/\sigma+\frac{\beta}{\Sx/\sigma})}\cdot   \sqrt{\frac{(\Sx/\sigma)^4-\beta}{(\Sx/\sigma)^4+\beta\cdot (\Sx/\sigma)^2}}
	\\
	&=&2\cdot \Sx^2+\sigma^2+\beta\cdot \sigma^2 +\frac{\beta\cdot \sigma^4}{\Sx^2} \notag\\
	&-&2 \cdot \sigma^2 \cdot \sqrt{\big((\Sx/\sigma)^2+1\big) \big((\Sx/\sigma)^2+\beta\big)}\cdot \sqrt{\frac{(\Sx/\sigma)^4-\beta}{(\Sx/\sigma)^2 \cdot\big((\Sx/\sigma)^2+\beta\big)}}
	\\ \\
	&=&2\cdot \Sx^2+\sigma^2+\beta\cdot \sigma^2 +\frac{\beta\cdot \sigma^4} {\Sx^2} -2 \cdot\sigma^2 \cdot \sqrt{\frac{\big((\Sx/\sigma)^2+1\big)\cdot\big((\Sx/\sigma)^4-\beta\big)} {(\Sx/\sigma)^2}} \\ 
	&=&2\cdot \Sx^2+\sigma^2+\beta\cdot \sigma^2 +\frac{\beta\cdot \sigma^4}{\Sx^2} 	-2 \cdot \sqrt{\frac{(\Sx^2+\sigma^2)\cdot(\Sx^4-\beta\cdot \sigma^4)}
		{\Sx^2}}
	\\
	&=&\Sx^2+\sigma^2\cdot\big((\Sx/\sigma+\frac{1}{\Sx/\sigma})\cdot(\Sx/\sigma+\frac{\beta}{\Sx/\sigma})\big) \notag\\
	&-&2\cdot \Sx\cdot \sigma\cdot \sqrt{(\Sx/\sigma +\frac{1}{\Sx/\sigma})(\Sx/\sigma+\frac{\beta}{\Sx/\sigma})}\cdot   \sqrt{\frac{(\Sx/\sigma)^4-\beta}{(\Sx/\sigma)^4+\beta\cdot (\Sx/\sigma)^2}}
\\ \\
	&=&2\cdot \Sx^2+\sigma^2+\beta\cdot \sigma^2 +\frac{\beta\cdot \sigma^4}{\Sx^2} \notag\\
	&-&2 \cdot \sigma^2 \cdot \sqrt{\big((\Sx/\sigma)^2+1\big) \big((\Sx/\sigma)^2+\beta\big)}\cdot \sqrt{\frac{(\Sx/\sigma)^4-\beta}{(\Sx/\sigma)^2 \cdot\big((\Sx/\sigma)^2+\beta\big)}}
\\
	&=& 2\cdot \Sx^2+\sigma^2+\beta\cdot \sigma^2 +\frac{\beta\cdot \sigma^4} {\Sx^2}\notag\\
	&-&2 \cdot\sigma^2 \cdot \sqrt{\frac{\big((\Sx/\sigma)^2+1\big)\cdot\big((\Sx/\sigma)^4-\beta\big)} {(\Sx/\sigma)^2}} \\
	&=& 2\cdot \Sx^2+\sigma^2+\beta\cdot \sigma^2 +\frac{\beta\cdot \sigma^4}{\Sx^2} 	-2 \cdot \sqrt{\frac{(\Sx^2+\sigma^2)\cdot(\Sx^4-\beta\cdot \sigma^4)}
		{\Sx^2}} \\
	 &=& \Sx^2+\sigma^2+\frac{\Sx^4-\beta\cdot \sigma^4}{\Sx^2}+\beta\cdot \sigma^2 +\frac{2\cdot\beta\cdot \sigma^4}{\Sx^2} \notag\\
	&-&2 \cdot \sqrt{\frac{(\Sx^2+\sigma^2)\cdot(\Sx^4-\beta\cdot \sigma^4)} {\Sx^2}} \\
	\\[8pt]
	&=& \left(\sqrt{\Sx^2+\sigma^2}-\sqrt{\frac{\Sx^4-\beta\cdot\sigma^4}{\Sx^2}}\right)^2 +\beta\cdot \sigma^2+\frac{2\cdot\beta\cdot \sigma^4}{\Sx^2}
	\end{eqnarray*}	
}
\end{proof}

\begin{proof}[Proof of Theorem \ref{thmThmMDSExplicitLoss}]
The result now follows from 
 Lemma \ref{lem:MDSLossLemma} and Lemma 
 \ref{lem:MDSLossLemmaOverBulk}.
\end{proof}

\subsection{Optimal SVHT Asymptotic Loss}
\begin{lem}
\label{lem:LossSeperateSigValsLemma}
	 Let $d\in \N$. For any matrix $\Mx\in M_{n\times d}$, without any degenerate positive singular values, and any monotone increasing shrinker $\eta:\R\rightarrow \R$, which holds $\eta(\Shy)=0$ for any $\Shy\leq \sigma\cdot (1+\sqrt{\beta})$ , the $\eta$ estimator has the following loss:   
\begin{eqnarray}
\label{eq: EtaAsLoss}
L(\Mxh^{\eta}|\textbf{\Sx}) \overset{a.s.}{=}
\sum_{i=1}^{n} L_{\eta}\Big(\eta\big(\Shy(x_i)\big),\Sx_i\Big)
\end{eqnarray}
where
$
L_{\eta}(\Sxh_i,\Sx_i) =
\left\{\begin{array}{ll}
\hat{\Sx}_i^2+\Sx_i^2-2\cdot \Sx_i\cdot \hat{\Sx}_i\cdot c(\Sx_i) & 
{\ }if {\ } \Sx_i>\sigma\cdot \beta^{1/4}\\
\Sx_i^2 \cdot 1_{[i\leq d]}  & otherwise
\end{array}
\right.
$ \\
and $t=\# \{i\in [d]\ : x_i>\sigma\cdot \beta^{1/4} \}$.
\end{lem}

\begin{proof}
	The loss of any monotone increasing shrinker, as discussed in subsection \ref{Shrinkage}, over some specific n is
\begin{eqnarray*}
M_n^2(\Mxh^{\eta},\Mx_n)
&=&
\underset{R\in O(p_n)}{min}
\| \sum_{i=1}^{t} \eta(y_{n,i})\cdot q_{n,i} \cdot \V{u_{n,i}} \cdot \V{e_i}^\T\cdot R - \sum_{j=1}^d x_j \cdot \V{v_{n,j}} \cdot \V{\tilde{v}_{n,j}} ^\T \|_{F}^2 \\
&=& \underset{R\in O\left(p_n \right)}{min} \left(\sum_{i=1}^{d} \Sx_i^2\right) +\left(\sum_{j=1}^{t} \eta(\Shy_{n,j})^2\right) \\
&-&2\cdot \sum_{i=1}^{t} \sum_{j=1}^{d} \Sx_j \cdot \eta(\Shy_{n,i})\cdot q_{n,i}\cdot  \langle\V{u_{n,i}}, \V{v_{n,j}}\rangle \langle R^{\T}\cdot \V{e_i} ,\V{\tilde{v}_{n,j}}\rangle
\end{eqnarray*}
where $e_i\in \R^{p_n}$ and $q_{n,i}\in \{\pm 1\}$ for any n and i. By using the same derivation as in lemma \ref{lem:MDSLossLemma}  and
assuming that $\eta$ is a monotone increasing function, we get
\begin{eqnarray*}
L(\Mxh^{\eta}|\V{\Sx})&=& \underset{n\rightarrow \infty}{lim} M_n^2(\Mxh^{\eta},\Mx_n)\\
& \overset{a.s.}{=}&\left(\sum_{i=1}^{t} \eta\big(\Shy(x_i)\big)^2+\Sx_i^2-2\cdot \Sx_i\cdot \eta\big(\Shy(\Sx_i)\big)\cdot c(\Sx_i)\right)+ \sum_{j={t}+1}^d \Sx_j ^2
\end{eqnarray*}
where the minimizer R is the same as in \ref{lem:MDSLossLemma}.
\end{proof}
~\\
The following Lemma is elementary:
\begin{lem}
\label{lem:mdsSVHTShrinkerFunctionLemma}
	Given a function $f:\mathbb{R}\rightarrow\mathbb{R}$, which has the following properties:
	\begin{enumerate}
		\item It is a 3rd degree polynomial, with three real roots.
		\item $\underset{x\rightarrow s \cdot \infty}{lim}f(x)=-s \cdot\infty$ for any $s\in \{\pm 1\}$
		\item It has exactly two real critical points ($\frac{\partial f}{\partial x}=0$), where
	\end{enumerate} 
	Then $f(a)<0<f(b)$ and there exists a single root of f in each of the following domains: $(-\infty,a),(a,b),(b,\infty)$
\end{lem}
\noindent We are now ready to prove our next main result.
\begin{proof}[Proof of Theorem \ref{thmmdsSVHTShrinker}]
	Following lemma \ref{lem:LossSeperateSigValsLemma}, the optimal SVHT estimator should minimize $L_{\eta}\big(y(x),x\big)$. As we consider the threshold estimators we would like to find when $L_{\eta}\big(y(x),x\big)\leq L_{\eta}(0,x)$, for any $y(x)\geq \sigma \cdot \beta^{1/4}$

\begin{eqnarray*}
\Shy(\Sx)^2+\Sx^2-2\cdot \Sx\cdot \Shy(\Sx)\cdot c(\Sx) &\leq& \Sx^2 \\
\Shy(\Sx)-2\cdot \Sx\cdot c(\Sx)&\leq& 0
\end{eqnarray*}
Following lemma \ref{lem:MPPropHY}, the following properties should hold $x \geq \sigma\cdot\beta^{1/4}$ and $c(x)\geq 0$. 
By plugging-in their definitions we get
\begin{eqnarray*}
0&\geq&\sigma\cdot \sqrt{(\Sx/\sigma+\frac{1}{\Sx/\sigma})(\Sx/\sigma+\frac{\beta}{\Sx/\sigma})}
-2\cdot \Sx\cdot \sqrt{\frac{(\Sx/\sigma)^4-\beta}{(\Sx/\sigma)^4+\beta\cdot(\Sx/\sigma)^2 }}\\
&=&\frac{\sigma}{\Sx/\sigma}\cdot \sqrt{\big((\Sx/\sigma)^2+1\big)\big((\Sx/\sigma)^2+\beta\big)}
-2\cdot \sigma \cdot \sqrt{\frac{(\Sx/\sigma)^4-\beta}{(\Sx/\sigma)^2+\beta }}\\
&=& \frac{\sigma}{x/\sigma \cdot\sqrt{(x/\sigma)^2+\beta}}\left(
\left( \left(x/\sigma\right)^2 +\beta\right)\sqrt{(x/\sigma)^2+1}
-2(x/\sigma)\sqrt{(x/\sigma)^4-\beta}
\right)
\end{eqnarray*}
Implying
\begin{eqnarray*}
0& \geq \left( (x/\sigma)^2+\beta\right)^2\left((x/\sigma)^2+1 \right)-4(x/\sigma)^2\left((x/\sigma)^4-\beta \right)\\
\end{eqnarray*}
or equivalently
\begin{eqnarray}
\label{eq:SVHTShrinkXEq}
-3\cdot(\Sx/\sigma)^6+(\Sx/\sigma)^4\cdot(1+2\cdot\beta)+(\Sx/\sigma)^2\cdot(\beta^2+6\beta)+\beta^2\leq
0\,.
\end{eqnarray}
\\
We would like to show that there exists a unique threshold value \textit{a} where for any $x<a$ holds  $L_{\eta}(x,y(x))\geq L_{\eta}(x,0)$ and for any $x>a$ holds $L_{\eta}(x,y(x))\leq L_{\eta}(x,0)$. \\

In order to show that there exists exactly one real positive solution we define: $z\equiv(\Sx/\sigma)^2$, and we would like to show that:
\begin{eqnarray}
\label{eq:SVHTProofSinglePosRoot}
f(z)=-3\cdot z^3+ (2\beta+1)\cdot z^2+(\beta^2+6\cdot\beta)\cdot z+\beta^2
\end{eqnarray}
has only one real positive root.\\

f(z) has three real distinct roots, since its Discriminant is positive-
\begin{eqnarray*}
\Delta_3&=&  
(2\beta+1)^2\cdot (\beta^2+6\beta )^2-4\cdot (-3)\cdot (\beta^2+6\beta)^3-4\cdot(2\beta+1)^3\cdot(\beta^2) \notag\\
&-&27\cdot(-3)^2\cdot(\beta^2)^2+18\cdot(-3)\cdot(2\beta+1)\cdot(\beta^2+6\beta)\cdot \beta^2\\
 \notag\\
&=& \beta^2\cdot\big(
(4\beta^2+4\beta+1)\cdot(\beta^2+12\beta+36)+12\cdot\beta(\beta^3+18\beta^2+108\beta+216) \notag\\
&-& 4\cdot(8\beta^3+12\beta^2+6\beta+1)-243\beta^2-54\cdot(2\beta^3+13\beta^2+6\beta)
\\
 \notag\\
&=&\beta^2\cdot\big(
(4\beta^4+52\beta^3+193\beta^2+156\beta+36)+
(12\beta^4+216\beta^3+1296\beta^2+2592\beta) \notag\\
&-& (32\beta^3+48\beta^2+24\beta+4)
-243\beta^2
-(108\beta^3+702\beta^2+324\beta)\big)
\\
 \notag\\
&=& \beta^2\cdot\big(
16\beta^4+128\beta^3+496\beta^2+2400\beta+32\big) \\
&>&0
\end{eqnarray*}
where $\beta \in (0,1]$. Now, we would like to show that only one root is positive. Therefore we start by finding the extreme point of the function-\\
\[
0=\frac{\partial f(z)}{z}=-9\cdot z^2+2(2\cdot \beta+1)\cdot z+\beta^2+6\cdot \beta 
\]

Its roots are
\[
\label{eq:SVHTProofRoots}
z_{1,2}=\frac{2\cdot \beta+1 \pm \sqrt{(2\cdot \beta+1 )^2 +9(\beta^2+6\cdot \beta)}}{9}
\]

We can see that $z_1<0<z_2$. Using lemma \ref{lem:mdsSVHTShrinkerFunctionLemma} we can deduce that $f(z_1)<0<f(z_2)$.\\

Moreover, following the lemma we know that there exists only one root in the domain of $(z_1,z_2)$. Since $f(0)=\beta^2$, the root is located in $(z_1,0)$, meaning that f contains only one positive root. \\

Now, if we describe $f(z)$ using $x$ instead of $z$, then the function would have 
six roots - one strictly positive, one strictly negative and four complex. Moreover, the value of f between zero and the positive root would be strictly positive, while after that point it would become strictly negative. Meaning that there exists a unique threshold that minimizes $L_\eta (\eta(y(x)),x)$ over any x, where $\eta(y)=y\cdot 1_{[a<y]}$.\\

Now will show that the real positive root is bigger than $\sigma \cdot \beta ^{1/4}$, by showing that $z_2>\beta ^{\frac{1}{2}}$ since $z_2$ is smaller than the positive root
\begin{eqnarray*}
\frac{2\cdot \beta+1 + \sqrt{(2\cdot \beta+1 )^2 +9(\beta^2+6\cdot \beta)}}{9}&>&\beta^{1/2}
\\
2\cdot \beta+1 + \sqrt{(2\cdot \beta+1 )^2 +9(\beta^2+6\cdot \beta)}&>&9\cdot\beta^{1/2}
\\
\sqrt{(2\cdot \beta+1 )^2 +9(\beta^2+6\cdot \beta)}&>&9\cdot\beta^{1/2}-2\cdot \beta-1
\end{eqnarray*}

The LHS is strictly positive while the RHS of the equation has two roots $\beta=\frac{(9\pm \sqrt{73})^2}{16}$. Meaning that we need to validate the inequality over the following domains $\beta$ - $(0,\frac{(9- \sqrt{73})^2}{16})$ and $(\frac{(9- \sqrt{73})^2}{16}),1]$ as we only consider $\beta \in(0,1]$.
The RHS is continuous over $\beta\in(0,1]$, meaning that by sampling a single point in each of the domains we can see that it is negative in $(0,\frac{(9- \sqrt{73})^2}{16})$, and positive in the second domain. Meaning that the inequality holds for any $\beta\in (0,\frac{(9-\sqrt{73})^2}{16}]$.\\

As for the second domain, both sides of the inequality are positive meaning that we can take a square of each of the sides while keeping the inequality:\\
\begin{eqnarray*}
(2\cdot \beta+1 )^2 +9(\beta^2+6\cdot \beta)&>&(9\cdot\beta^{1/2}-2\cdot \beta-1)^2\\
13\cdot \beta^2+58\cdot \beta+1&>& 81\cdot \beta +4\cdot \beta^2 +1 -36\cdot \beta^{3/2}+4\cdot\beta-18\cdot \beta^{1/2} \\
9\cdot \beta^2+36\cdot \beta^{3/2}-27\cdot \beta+18\cdot \beta^{1/2}&>&0
\end{eqnarray*}
It is positive since
\begin{eqnarray*}
36\cdot \beta^{3/2}-27\cdot \beta+18\cdot \beta^{1/2}=&
9\cdot\beta^{1/2}(4\cdot\beta-3\cdot \beta^{1/2}+2)\\
=&9\cdot\beta^{1/2}\Big((2\cdot\beta^{1/2}-\frac{3}{4})^2-\frac{9}{16}+2\Big)\\
>&0
\end{eqnarray*}
\\
Following lemma \ref{lem:MPPropHY} the threshold should be strictly bigger then $\lambda>\sigma\cdot (1+\sqrt{\beta})$ when applied on the noisy singular values.

\end{proof}

\subsection{Optimal Continuous Shrinker Asymptotic Loss}
 
\begin{proof}[Proof of Theorem \ref{thmmdsOptShrinkerExplicit}]
Following lemma \ref{lem:LossSeperateSigValsLemma}, the optimal continuous shrinker should minimize $L_{\eta}\left(\eta \left( \Shy \left( \Sx \right) \right),\Sx\right)$ for $y \geq \sigma \cdot (1+\sqrt{\beta})$
\begin{eqnarray*}
L_{\eta}\left(\eta\left( \Shy \left( \Sx \right)\right),\Sx \right) =
\eta\left( \Shy \left( \Sx \right) \right)^2 + x^2 -2\cdot x\cdot \eta\left( \Shy \left( \Sx \right) \right) \cdot c(\Sx)
\end{eqnarray*}
\\
while for $y<\sigma\cdot (1+\sqrt{\beta})$ it should hold $\eta(\Shy)=0$. Now
since $L_{\eta}$ is convex in $\eta\left(y \left( \Sx \right) \right)$ the
optimal shrinker  satisfies
\begin{eqnarray*}
0=\frac{\partial}{\partial \eta(\Shy\left( \Sx \right))} \big( L_{\eta}(\eta\left(\Shy\left(\Sx \right)\right),\Sx)\big)=2\cdot \eta\left(\Shy \left( \Sx \right) \right)-2\cdot \Sx \cdot c(x)
\end{eqnarray*}
meaning that
\begin{eqnarray}
\label{optEta_form:eq}
\eta^*\left(y \left( \Sx \right) \right)=x\cdot c(x)\,.
\end{eqnarray}
\noindent Following lemma \ref{lem:MPPropHY} $x \geq \sigma\cdot\beta^{1/4}$ and $c(x)\geq 0$. 
Next we plug-in the definitions of $y(x)$ and $c(x)$
\begin{eqnarray}
\eta^*\left(\Shy \left( \Sx \right) \right) &=&
\Sx \cdot \sqrt{\frac{(\Sx/\sigma)^4-\beta}{(\Sx/\sigma)^4+\beta\cdot (\Sx/\sigma)^2}}\nonumber\\
&=&\sigma\cdot\sqrt{\frac{(\Sx/\sigma)^4-\beta}{(\Sx/\sigma)^2+\beta}} \nonumber\\
\label{eq:thmmdsOptShrinkerExplicit}
&=&\sqrt{\frac{\Sx^4-\beta\cdot \sigma^4}{\Sx^2+\beta\cdot \sigma^2}} \\
&=&\sqrt{\Sx^2-\frac{\Sx^2\cdot \beta\cdot \sigma^2+\beta\cdot \sigma ^4}{\Sx^2+\beta\cdot \sigma^2}} \nonumber\\
&=&\sqrt{\Sx^2-\beta\cdot \sigma ^2\frac{\Sx^2+\sigma ^2}{\Sx^2+\beta\cdot \sigma^2}}\nonumber\\ 
&=&\sqrt{\Sx^2-\beta\cdot \sigma ^2\cdot \left(1+\frac{\sigma ^2\cdot(1-\beta)}{\Sx^2+\beta\cdot \sigma^2}\right)}\nonumber\\
&=&\sqrt{\Sx^2-\beta\cdot \sigma ^2-\frac{\sigma ^4\cdot\beta \cdot(1-\beta)}{\Sx^2+\beta\cdot \sigma^2}}\nonumber\\
&=&\sigma\cdot \sqrt{(\Sx/\sigma)^2-\beta-\frac{\beta(1-\beta)}{(\frac{\Sx}{\sigma})^2+\beta}} \nonumber
\end{eqnarray}
$\eta^*(y(x))$ is continuous at $\Sx=\sigma\cdot\beta^{1/4}$ since %
\begin{eqnarray*}
lim_{\Sx\rightarrow(\sigma \cdot \beta^{1/4})^{-}} \eta_{opt} \big( \Shy(\Sx)\big)&=&0\\
lim_{\Sx\rightarrow(\sigma \cdot \beta^{1/4})^{+}} \eta_{opt} \big( \Shy(\Sx)\big)=\sqrt{\frac{0}{\Sx^2+\beta\cdot\sigma^2}}&=&0
\end{eqnarray*}%
Meaning that this optimal shrinker is indeed continuous.
\end{proof}
\begin{proof}[Proof of Theorem \ref{thmMDSPLoss}]
Following lemma \ref{lem:LossSeperateSigValsLemma} and \eqref{optEta_form:eq}, the loss could be expressed as follows
\begin{eqnarray}
L(\Mxh^{\eta^*}| \V{\Sx}) &\overset{a.s.}{=}&
\mysum{i}{1}{t} \eta(\Shy(\Sx_i))^2+\Sx_i^2-2\cdot \Sx_i\cdot \eta\big(\Shy(\Sx_i)\big)\cdot c(\Sx_i)
+\mysum{i}{t+1}{d} \Sx_i^2 \nonumber\\
&=&
\mysum{i}{1}{t} \big(c(\Sx_i)\cdot \Sx\big)^2+\Sx_i^2-2\cdot \Sx_i\cdot \big(c(\Sx_i)\cdot \Sx_i\big)\cdot c(\Sx_i)
+\mysum{i}{t+1}{d} \Sx_i^2 \nonumber \\
&=&
\mysum{i}{1}{t} \Sx_i^2-c(\Sx_i)^2 \cdot \Sx_i^2+\mysum{i}{t+1}{d} \Sx_i^2 \nonumber \\
&=&
\mysum{i}{1}{t} \Sx_i^2\cdot \big(1-c(\Sx_i)^2\big)+\mysum{i}{t+1}{d} \Sx_i^2  \nonumber \\
\label{eq:MDSPProofeq1}
&=&\mysum{i}{1}{t} \Sx_i^2\cdot \left(1-\frac{(\Sx_i/\sigma)^4-\beta}{(\Sx_i/\sigma)^4+\beta\cdot (\Sx_i/\sigma)^2}\right)+\mysum{i}{t+1}{d} \Sx_i^2 \\
&=& \mysum{i}{1}{t} \Sx_i^2\cdot \left(\frac{\beta\cdot (\Sx_i/\sigma)^2+\beta}{(\Sx_i/\sigma)^4+\beta\cdot (\Sx_i/\sigma)^2}\right)+\mysum{i}{t+1}{d} \Sx_i^2 \nonumber \\
&=& \mysum{i}{1}{t} \beta\cdot \sigma^2 \cdot \frac{(\Sx_i/\sigma)^2+1}{(\Sx_i/\sigma)^2+\beta}+\mysum{i}{t+1}{d} \Sx_i^2 \nonumber \\
&=&\beta\cdot \sigma^2 \left(\mysum{i}{1}{t} \frac{1-\beta}{(\Sx_i/\sigma)^2+\beta} +1 \right) +\mysum{i}{t+1}{d} \Sx_i^2 \nonumber
\end{eqnarray}
\end{proof}
\begin{lem} 
	Let $r,d\in \N$ and $\Mx\in M_{n\times d}$, without any degenerate positive singular values. Denote $\delta _i \equiv L_{TSVD}(\Sx^{r}_i,\Sx_i)-L_{\eta}(\eta^*(\Shy(\Sx_i)),\Sx_i)$. 
	\label{lem:OptLossBetterTSVDLoss}

	\begin{enumerate}
		\item if $\Sx_i\in (0,\sigma\cdot \beta^{1/4})$ then
		\begin{eqnarray}
		\delta_i =1_{[i\leq r]} \cdot \sigma^2\cdot (1+\sqrt{\beta})^2
		\end{eqnarray}
		\item if $\Sx_i\in (\sigma\cdot \beta^{1/4},\infty]$ and $i>r$ then
		\begin{eqnarray}
		\delta_i = \sigma^2 \cdot \frac{(x_i/\sigma)^4-\beta}{(x_i/\sigma)^2+\beta}
		\end{eqnarray}
		\item if $\Sx_i\in (\sigma\cdot \beta^{1/4},\infty)$ and $i\leq r$ then
		\begin{eqnarray}
		\delta_i =	\left(\sqrt{x_i^2+\sigma^2}-\sqrt{\frac{x_i^4-\beta\cdot \sigma^4}{x_i^2}}\right)^2 +\beta\cdot \sigma^2\cdot 
		\frac{(x_i/\sigma)^2\cdot (1+\beta)+2\beta}{(x_i/\sigma)^4+\beta \cdot(x_i/\sigma)^2}
		\end{eqnarray}		
	\end{enumerate}
where $\Sxh^{r}_i$ is the i-th singular value of the embedding done by the TSVD embedding algorithm $\Mxh^{r}$.
\end{lem}

\begin{proof}
	Following lemma \ref{lem:MDSLossLemma}, \ref{lem:MDSLossLemmaOverBulk},\ref{lem:LossSeperateSigValsLemma}, and theorem \ref{thmMDSPLoss}
	the following domains should be investigated over $\Sx_i$: $[0,\sigma\cdot \beta^{1/4}]$, $[\sigma\cdot \beta^{1/4}$, $\infty]$ when $i>r$ ,and when $i\leq r$. \\
	First we examine the domain of $(0,\sigma\cdot (1+\sqrt{\beta}))$
	\begin{eqnarray*}
	\delta_i&=&
	\sigma^2\cdot (1+\sqrt{\beta})^2\cdot 1_{[i\leq r]} + x_i^2 \cdot 1_{[i\leq d]}-x_i^2 \cdot 1_{[i\leq d]} \\ \\
	&=& \sigma^2\cdot (1+\sqrt{\beta})^2\cdot 1_{[i\leq r]}
	\end{eqnarray*}
	Following equation \eqref{eq:MDSPProofeq1}, the difference between the two functions for $x_i\in (\sigma\cdot \beta^{1/4},\infty)$ when $i>r$ is
	\begin{eqnarray*}
		\delta_i&=&
		\Sx_i^2- \Sx_i^2\cdot \left(1-\frac{(\Sx_i/\sigma)^4-\beta}{(\Sx_i/\sigma)^4+\beta\cdot (\Sx_i/\sigma)^2}\right)
		\\
		&=&
	\sigma^2 \cdot \frac{(x_i/\sigma)^4-\beta}{(x_i/\sigma)^2+\beta}
	\end{eqnarray*}
	While when $i\leq r$ it is
	\begin{eqnarray*}
	\delta_i&=&
	\Bigg( \left(\sqrt{x_i^2+\sigma^2}-\sqrt{\frac{x_i^4-\beta\cdot \sigma^4}{x_i^2}}\right)^2 
	+ \frac{2\cdot \beta \cdot \sigma^4}{x_i^2}+\beta\cdot \sigma^2
	\Bigg)\\
	&-& \beta\cdot \sigma^2 \cdot \left( \frac{1-\beta}{(x_i/\sigma)^2+\beta}+1 \right)\\
	\\
	&=&\left(\sqrt{x_i^2+\sigma^2}-\sqrt{\frac{x_i^4-\beta\cdot \sigma^4}{x_i^2}}\right)^2 
	+\beta\cdot \sigma^2\cdot 
	\left(\frac{2}{(x_i/\sigma)^2}-\frac{1-\beta}{(x_i/\sigma)^2+\beta}\right)
	\\ \\
	&=&\left(\sqrt{x_i^2+\sigma^2}-\sqrt{\frac{x_i^4-\beta\cdot \sigma^4}{x_i^2}}\right)^2 \\ 
	&+&\beta\cdot \sigma^2\cdot 
	\frac{(x_i/\sigma)^2\cdot (1+\beta)+2\beta}{(x_i/\sigma)^4+\beta \cdot(x_i/\sigma)^2}
	\end{eqnarray*}
\end{proof}

\begin{proof}[Proof of Theorem \ref{thmOptBetterTSVD}]
    The result follows from 
    Lemma \ref{lem:MDSLossLemma}, 
    Lemma \ref{lem:LossSeperateSigValsLemma} and
    Lemma \ref{lem:OptLossBetterTSVDLoss}.
\end{proof}


\begin{proof}[Proof of Theorem \ref{thmSigmaEstimation}]
    Denote the singular values of $H\cdot Y_n$ by $y_{1,n}\geq y_{n,n}\geq 0$.
		A similarity matrix $S_n$ is defined through it corresponding $\Delta_n$, as shown in \eqref{eq:S}. Denote $S_n$'s eigenvalues by $s_{n,1}\geq s_{n,n}\geq 0$. Theorem \ref{thmMDSRealLoss} assembles the following connection $S_n=(H\cdot Y_n)\cdot (H\cdot Y_n)^\T$. As a consequence,  $s_{n,i}=\Shy_{n,i}^2$ holds. \\
	
	Let $F_n$ be the Cumulative Empirical Spectral Distribution of $S_n$, and let $Median(\cdot)$ be the functional that extracts the median out of any Cumulative distribution function. 
	\[
		 Median(F_n)=y_{n,med}
	\] 
	
	By Lemma \ref{lem:MPPropHY}, all but the $d$ largest 
	singular values $\{\Shy_{i}\}$
	asymptotically follow the Quarter Circle distribution  \eqref{eq:QC}
	as $n\rightarrow\infty$. Therefore the eigenvalues of S act as the Marcenko Pastur distribution \cite{Marcenko1967}.\\

	We denote the Cumulative Empirical Spectral Distribution of $S/\sigma^2 $ by $\tilde{F}_n$, making the effective noise level on it to be 1. Under our asymptotic framework, almost surely, $\tilde{F}_n$ converges weakly to a limiting distribution, $F_{MP}$, the CDF of the Marceno Pastur distribution with shape parameter $\beta$ \cite{bai2010spectral}. The median functional is continuous for weak convergence at $F_{MP}$, therefore 
	\[
		\underset{n\rightarrow\infty }{lim}\frac{s_{n,med}}{\sigma^2}=\underset{n\rightarrow\infty }{lim} Median(\tilde{F}_n) \overset{a.s.}{=} Median(F_{MP})=  \mu_{\beta}
	\]
	By that we can conclude the almost surely convergence of our estimator
	\begin{eqnarray}
	    \underset{n \rightarrow \infty}{lim}
	    \hat{\sigma}^2(S_n)=
		\underset{n \rightarrow \infty}{lim} \frac{s_{n,med}}{\mu_\beta} \overset{a.s.}{\rightarrow}\sigma^2
	\end{eqnarray}
\end{proof}

\section{Conclusion}
This paper presents a systematic treatment of Multidimensional Scaling (MDS)
from a decision-theoretic perspective. By introducing a loss function which
measures the embedding accuracy, and introducing a useful asymptotic model 
we were able to derive an asymptotically precise selection rule for the
embedding dimension, as well as a new version of MDS which uses an optimal
shrinkage non-linearity, under the assumption of white measurement 
noise. The proposed algorithm is no more complicated to implement than classical
MDS, yet offers significant improvement in performance as measured by the
asymptotic loss. Our results indicate that manifold learning algorithms are
inherently sensitive to ambient noise in high dimensions, a phenomenon that
calls for further study. 

\section*{Acknowledgements}
The authors thank Zohar Yachini and Shay Ben Elazar for fascinating discussions on
applications of MDS. 
This work was partially supported by 
H-CSRC Security Research Center, 
Israeli Science Foundation
    grant no. 1523/16 and
German-Israeli foundation for scientific research and development (GIF)
    Program no. I-1100-407.1-2015.

\appendix
\section*{Appendix: The case $\beta>1$}

Surprisingly we can see that all of our results admit the case of $\beta \geq 1$ as well. We give a sketch the proof of lemma \ref{lem:MPPropHY_bigBeta} that is the basis for analyzing the loss function under the new $\beta$ domain. Moreover, we show that the optimal threshold is still bigger than the bulk edge ($\sigma\cdot (1+\sqrt{\beta})$). We finish by showing how one should find $\sigma$. \\
One important feature under this configuration is that $rank(S)=p$, meaning that the algorithm can infer the ambient dimension of the data under this configuration from the distance matrix $\Delta$.
\begin{lem}[Variation of Lemma \ref{lem:MPPropHY}]
	\label{lem:MPPropHY_bigBeta}
	\begin{enumerate}\text{ }
		\item 	\begin{eqnarray}
		\underset{n\rightarrow\infty}{lim}\Shy_{n,i}\overset{a.s.}{=}
		\left\{ \begin{array}{ll}
		\Shy(\Sx) & if {\ }i\in[t] \\ \sigma\cdot (1+\sqrt{\beta}) &
		otherwise
		\end{array} \right.
		\end{eqnarray} 
		\item \begin{eqnarray}
		\underset{n\rightarrow\infty}{lim}| \langle \V{u_{n,i}},\V{v_{n,j}} \rangle | \overset{a.s.}{=} \left\{ \begin{array}{ll}
		c(\Sx) & i\in[t]\\
		0 & otherwise
		\end{array} \right.
		\end{eqnarray}
		
	\end{enumerate}
	where 
	$t\equiv \#\{ i\in [d] : x_i/ \sigma >  \beta^{1/4} \}$
\end{lem}

\begin{proof}
	In order to follow lemma \ref{lem:MPPropHY}'s proof, one should make some adaptations. Denote $\tilde{\Mx}_n\equiv \Mx_n^\T$ and $\tilde{\My}_n\equiv \My_n^\T$, meaning
	\[
	\tilde{\My}_n= R_n^\T \left[\begin{array}{c}
	\tilde{\Mx}_n \\ 0_{(p_n-d) \times n} \end{array} \right] +Z_n^\T
	\]
	Denote its aspect ratio $\tilde{\beta}\equiv\frac{1}{\beta}=\frac{p_n}{n-1}$ and its effective level $\tilde{\sigma}\equiv \sigma/\sqrt{\tilde{\beta}}$. Denote the new noise matrix by $\tilde{Z}_n\equiv Z_n^\T$. The entries of $Z_n$ are i.i.d distributed and drawn from a distribution with zero mean, variance $\tilde{\sigma}/\sqrt{n-1}= \sigma/\sqrt{p_n}$ and finite forth moment. As one can suspect the new notations admit the asymptotic model framework that was stated in subsection \ref{AsymptoticModel}. Combining all those notation together, the framework could be written as
	\[
	\tilde{\My}_n= R_n^\T \left[\begin{array}{c}
	\tilde{\Mx}_n \\ 0_{(p_n-d) \times n}  \end{array} \right] +\tilde{Z}_n
	\]
	with the aspect ratio $\tilde{\beta}\in (0,1]$.\\
	Now, by following carefully after the proof of lemma \ref{lem:MPPropHY}.
\end{proof}

\begin{proof}[Proof of the SVHT]
	Up until the point where we show that the threshold is bigger then $\sigma\cdot \beta^{1/4}$, the proof is identical. For the this point we show 
	\begin{eqnarray*}
		\frac{2\cdot \beta+1 + \sqrt{(2\cdot \beta+1 )^2 +9(\beta^2+6\cdot \beta)}}{9}&>&\beta^{1/2}
		\\
		2\cdot \beta+1 + \sqrt{(2\cdot \beta+1 )^2 +9(\beta^2+6\cdot \beta)}&>&9\cdot\beta^{1/2}
		\\
		\sqrt{(2\cdot \beta+1 )^2 +9(\beta^2+6\cdot \beta)}&>&9\cdot\beta^{1/2}-2\cdot \beta-1
	\end{eqnarray*}
	
	The LHS is strictly positive while the RHS of the equation has two roots $\beta=\frac{(9\pm \sqrt{73})^2}{16}$. Meaning that we need to validate the inequality over the following domains $\beta$ - $[1,\frac{(9+ \sqrt{73})^2}{16})$ and $[\frac{(9+ \sqrt{73})^2}{16}),\infty)$ as we only consider $\beta \in[1,\infty)$.
	The RHS is continuous over $\beta\in[1,\infty)$, meaning that by sampling a single point in each of the domains we can see that it is positive in $[1,\frac{(9+ \sqrt{73})^2}{16})$, and negative in the second domain. Meaning that the inequality holds for any $\beta\in [\frac{(9+ \sqrt{73})^2}{16}),\infty)$.\\
	
	As for the first domain, both sides of the inequality are positive meaning that we can take a square of each of the sides while keeping the inequality:\\
	\begin{eqnarray*}
		(2\cdot \beta+1 )^2 +9(\beta^2+6\cdot \beta)&>&(9\cdot\beta^{1/2}-2\cdot \beta-1)^2\\
		13\cdot \beta^2+58\cdot \beta+1&>& 81\cdot \beta +4\cdot \beta^2 +1 -36\cdot \beta^{3/2}+4\cdot\beta-18\cdot \beta^{1/2} \\
		9\cdot \beta^2+36\cdot \beta^{3/2}-27\cdot \beta+18\cdot \beta^{1/2}&>&0
	\end{eqnarray*}
	It is positive since
	\begin{eqnarray*}
		36\cdot \beta^{3/2}-27\cdot \beta+18\cdot \beta^{1/2}&=&
		9\cdot\beta^{1/2}(4\cdot\beta-3\cdot \beta^{1/2}+2)\\
		&=&9\cdot\beta^{1/2}\Big((2\cdot\beta^{1/2}-\frac{3}{4})^2-\frac{9}{16}+2\Big)\\
		&>&0
	\end{eqnarray*}
	\\
	Following lemma \ref{lem:MPPropHY_bigBeta} the threshold should be strictly bigger then $\lambda>\sigma\cdot (1+\sqrt{\beta})$ when applied on the noisy singular values.\\
	
\end{proof}

\begin{thm}[Variation of Theorem \ref{thmSigmaEstimation}]
	Consider
\begin{eqnarray}
\hat{\sigma}(S)=\sqrt{\frac{s_{med}}{\mu_{1/\beta}\cdot \beta}} 
\end{eqnarray}
where $s_{1}\geq \ldots \geq s_{min(n,p)}\geq0 $ are the eigenvalues of S, and $s_{med}$ is their median. Denote the median of the quarter circle distribution \cite{bai2010spectral} for $\beta$ by $\mu_{\beta}$.
Then 
$ \hat{\sigma}^2(S_n)\overset{a.s.}{\longrightarrow} \sigma^2$ as
$n\to\infty$.
\end{thm}
	
	\bibliographystyle{unsrt}
	\bibliography{MDSP}

\end{document}